\documentclass[10pt]{article}

\usepackage{amssymb,amsmath,amsthm,amsfonts}
\usepackage{graphicx}
\usepackage{comment}
\usepackage{algorithm,algorithmic}
\usepackage{eepic,epic}
\usepackage{epsfig,subfigure}
\usepackage{verbatim}
\usepackage{bookmark}
\usepackage{indentfirst}

\numberwithin{equation}{section}

\begin{document}

\newtheorem{theorem}{Theorem}[section]
\newtheorem{lemma}{Lemma}[section]
\newtheorem{definition}{Definition}[section]
\newtheorem{example}{Example}[section]
\newtheorem{assumption}{Assumption}[section]
\newtheorem{remark}{Remark}[section]
\newtheorem{corollary}{Corollary}[section]
\numberwithin{equation}{section}

\newcommand{\al}{\alpha}
\newcommand{\fy}{\varphi}
\newcommand{\la}{\lambda}
\newcommand{\ep}{\epsilon}
\newcommand{\vtht}{{\widetilde \vartheta}}
\newcommand{\rlh}{{\widetilde \varrho}}

\def\dH#1{\dot H^{#1}(\Omega)}
\def\vecal{{\vec{\al}}}
\def\Dal{{\partial_t^\al}}
\def\dal{\bar \partial}
\def\Om{\Omega}
\def\II{(\Om)}
\def\bPtau{\bar\partial_\tau}
\def\wpsi{\widetilde \psi}
\def\hpsi{\widehat \psi}

\title{An Analysis of Galerkin Proper Orthogonal Decomposition for Subdiffusion}
\author{Bangti Jin\thanks{Department of Computer Science, University College London, Gower Street,
London WC1E 6BT, UK ({b.jin@ucl.ac.uk, bangti.jin@gmail.com})}
\and Zhi Zhou\thanks{Department of Applied Physics and Applied Mathematics,
Columbia University, New York, NY, 10027, USA
(zhizhou0125@gmail.com)}}
\date{\today}

\maketitle

\begin{abstract}
In this work, we develop a novel Galerkin-L1-POD scheme for the
subdiffusion model with a Caputo fractional derivative of order $\alpha\in (0,1)$ in time, which
is often used to describe anomalous diffusion processes in heterogeneous media. The
nonlocality of the fractional derivative requires storing all the solutions from time zero.
The proposed scheme is based on continuous piecewise linear finite elements, L1 time stepping,
and proper orthogonal decomposition (POD). By constructing an effective reduced-order scheme
using problem-adapted basis functions, it can significantly reduce the computational complexity and storage requirement.
We shall provide a complete error analysis of the scheme under realistic regularity
assumptions by means of a novel energy argument. Extensive numerical experiments are presented to verify
the convergence analysis and the efficiency of the proposed scheme.
\end{abstract}

\section{Introduction}\label{sec:intro}
In this work, we consider the following model initial-boundary value problem for $u(x,t)$:
\begin{alignat}{3}\label{eqn:fde}
   \Dal u-\Delta u&= f,&&\quad \text{in  } \Omega&&\quad T \ge t > 0,\notag\\
   u&=0,&&\quad\text{on}\  \partial\Omega&&\quad T \ge t > 0,\\
    u(0)&=v,&&\quad\text{in  }\Omega,&&\notag
\end{alignat}
where $\Omega$ is a bounded convex polygonal domain in $\mathbb R^d\,(d=1,2,3)$ with a boundary
$\partial\Omega$ and $v$ is a given function defined on the domain $\Om$ and $T>0$ is a fixed value.
Here  $\Dal u$ ($0<\al<1$) denotes the left-sided Caputo fractional derivative of order
$\al$ with respect to $t$ and it is defined by
(see, e.g. \cite[pp.\,91]{KilbasSrivastavaTrujillo:2006})
\begin{equation}\label{McT}
   \Dal u(t)= \frac{1}{\Gamma(1-\al)} \int_0^t(t-s)^{-\al}\frac{d}{ds}u(s)\, ds,
\end{equation}
where $\Gamma(\cdot)$ is Euler's Gamma function defined by $\Gamma(x)=\int_0^\infty s^{x-1}e^{-s}ds$ for $x>0$.

In recent years, the model \eqref{eqn:fde} has received much interest in physical modeling, mathematical analysis
and numerical simulation. The main engine that has fueled these developments is its extraordinary
capability for describing anomalously slow diffusion processes, in which the mean square
variance of particle displacements grows sublinearly with time, instead of
linear growth for a Gaussian process. At a microscopic level, the particle motion
is more adequately described by continuous time random walk, whose macroscopic counterpart is
a differential equation with a fractional derivative in time \cite{MetzlerKlafter:2000}. Nowadays the model has been
successfully employed in many applications, e.g., thermal diffusion in fractal domains
\cite{Nigmatulin:1986}, ion transport in column experiments \cite{HatanoHatano:1998}, and
non-Fickian transport in geological formation \cite{BerkowitzCortisDentzScher:2006}, to name just a few.

Numerically, the presence of the fractional derivative $\Dal u$ has two important consequences. First, the
nonlocality in time incurs huge storage requirement as well as much increased computational efforts along the
evolution of the time. Second, the solution operator has only very limited smoothing property: the problem
has at best order two smoothing in space \cite{SakamotoYamamoto:2011}, and the first derivative in time is
usually unbounded, cf. Theorem \ref{thm:reg-init} in the appendix.  These represent the main technical challenges in the
development and analysis of robust numerical schemes for reliably simulating subdiffusion. The challenges
are especially severe for ``multi-query'' applications, e.g., inverse problems and optimal control, where
repeated solutions of ``analogous'' forward problems are required, e.g., due to variation in problem parameters
or inputs. To reduce the storage requirement, a number of useful strategies have been proposed, e.g.,
short-memory principle and panel clustering \cite{Podlubny:1999,FordSimpson:2001,LopezLubich:2008,McLean:2012}.

In this work, we shall develop an efficient strategy, called the Galerkin-L1-POD scheme, for reliably simulating the
subdiffusion model \eqref{eqn:fde} by coupling the Galerkin finite element method (FEM) with proper orthogonal decomposition (POD)
to reduce the computational complexity of repeatedly simulating subdiffusion, which is
important for solving related inverse problems and optimal control.
POD is a popular model reduction technique, and it has achieved great success in reducing the
complexity of mathematical models governed by differential equations; see
\cite{KunischVolkein:2001,ChapelleGariahSainte:2012,Singler:2014,KunischVolkwein:2002siam,
AmsallemHetmaniuk:2014,SachsSchu:2013} for a rather incomplete list.
It is especially attractive in
optimal control \cite{KunischVolkwein:1999,HinzeVolkwein:2008,KunischVolkwein:2008,SachsVolkwein:2010} and parameter
inversion \cite{Jin:2008,MoireauChapelle:2011}. To the
best of our knowledge, this work represents the first application of the POD for the subdiffusion
model \eqref{eqn:fde} with a complete error analysis.

Next we describe the proposed scheme. Let $\mathcal{T}_h$ be a shape regular quasi-uniform partition of the
domain $\Omega$, and $X_h$ be the associated continuous piecewise linear finite element space. Meanwhile,
we discretize the Caputo fractional derivative $\Dal u(t)$ by the L1 approximation
$\bPtau^\al u(t_n)$ (with a time step size $\tau$) \cite{LinXu:2007,SunWu:2006}
\begin{equation*}
   \bPtau^\al u(t_n)=\sum_{j=0}^{n-1}b_j\frac{u(t_{n-j})-u(t_{n-j-1})}{\tau^\alpha\Gamma(2-\al)},
\end{equation*}
where the weights $\{b_j\}$ are defined by \eqref{eqn:b}.
With the Galerkin FEM in space and L1 approximation in time, we arrive at
the following fully discrete scheme: find $U_h^n\in X_h$ for $n=1,2,\ldots,N$
\begin{equation*}
   (\bPtau^\alpha U_h^n,\fy) + (\nabla U_h^n,\nabla \fy) = (f(t_n),\fy)\quad \forall \fy\in X_h,
\end{equation*}
with $U_h^0\in X_h$ being an approximation to the initial data $v$, where $(\cdot,\cdot)$ denotes
the $L^2(\Omega)$ inner product. The term $\bPtau^\alpha U_h^n$ involves all solutions $\{U_h^i\}_{i=0}^{n-1}$
preceding the current time step $n$, indicating the computational challenge. In this work,
we shall adopt the POD methodology to overcome the challenge. Specifically, we take the fully discrete solutions $\{U_h^n\}_{
n=0}^N$ and fractional difference quotients $\{\bPtau^\alpha U_h^n\}_{n=1}^N$ as snapshots to generate
an optimal orthonormal basis $\{\psi_j\}_{j=1}^r$. Since these snapshots are sampled from the solution
manifold, the POD basis is automatically adapted to the characteristics of the manifold and is expected
to have good approximation property. Then we employ a Galerkin framework using the POD space $X_h^m$, $m\leq r$,
spanned by the first $m$ POD basis functions, i.e., find $U_m^n\in X_h^m$, $n=1,2\ldots,N$ such that
\begin{equation*}
   (\bPtau^\alpha U_m^n,\fy) + (\nabla U_m^n,\nabla \fy) = (f(t_n),\fy)\quad \forall \fy\in X_h^m,
\end{equation*}
with $U_m^0\in X_h^m$ being an approximation to $U_h^0$. In the reduced order formulation, the degree of freedom is $m$, the
number of POD basis functions, which is usually much smaller than that of the full Galerkin formulation.
Hence, it yields an enormous reduction in computational complexity and storage requirement.
We shall provide a complete a priori convergence analysis
of the scheme. Our main theoretical result is given in Theorem \ref{thm:err-pod}. For example
for the POD approximation $\{U_m^n\}_{n=1}^N$ generated using the $H_0^1\II$-POD basis, the following
error estimate holds (with $\ell_h=|\log h|$)
\begin{equation*}
\frac{1}{N} \sum_{n=1}^N \| u(t_n)-U_m^n \|_{L^2\II}^2
\le c_T\big(\tau^{2\al} + h^4\ell_h^4  + \sum_{j=m+1}^r  \widetilde \la_j \big),
\end{equation*}
where $\{\widetilde\la_j\}_{j=1}^r$ are the descendingly ordered eigenvalues of the correlation matrix $\widetilde K$
(see Section \ref{ssec:POD} for details) under suitable verifiable regularity conditions on
the source term $f$ and the initial data $v$.

This error estimate consists of three components: spatial error $O(h^2\ell_h^2)$, temporal
error $O(\tau^{\alpha})$ and POD error $(\sum_{j=m+1}^r \widetilde \la_j)^{1/2}$. While
nearly optimal error estimates due to the  spatially semidiscrete Galerkin FEM is available
\cite{JinLazarovZhou:2013,JinLazarovPasciakZhou:2013a,JinLazarovLiuZhou:2015}, it is not
the case for temporal discretization by the L1 time stepping.
The L1 scheme was first analyzed in \cite{LinXu:2007,SunWu:2006}, where the local truncation error
was shown to be $O(\tau^{2-\al})$ for twice continuously differentiable (in time) solutions,
which is fairly restrictive, cf. Remark \ref{rmk:reg}. Recently  some error bounds that are
expressed directly in terms of data regularity for the homogeneous problem were shown using
a generating function approach \cite{JinLazarovZhou:L1}, however, the analysis does not
extend straightforwardly to the inhomogeneous case.

In this work we shall develop a novel energy argument for the L1 time stepping to overcome the technical challenge
in the convergence analysis, which represents the main technical novelty. We shall derive optimal error estimates under realistic regularity
conditions, and the analysis covers both smooth and nonsmooth problem data, cf. Theorem \ref{thm:error-fully}.
Further, the stability result plays an essential role in deriving error estimates due to the POD approximation.
All the theoretical results are fully confirmed by extensive numerical experiments.

The rest of the paper is organized as follows. In Section \ref{sec:scheme} we develop an efficient
 Galerkin-L1-POD scheme, and in Section \ref{sec:conv}, provide a complete error analysis of the scheme.
In Section \ref{sec:numeric}, extensive numerical experiments
for one- and two-dimensional examples are presented to verify the convergence analysis.
Finally, in an appendix, we briefly discuss the temporal regularity results for problem \eqref{eqn:fde}.
Throughout, the notation $c$, with or without a subscript, denotes
a generic constant, which may differ at different occurrences, but it is always independent of the
solution $u$, the mesh size $h$, time step size $\tau$, and the number $m$ of POD basis functions.


\section{An efficient Galerkin-L1-POD scheme}\label{sec:scheme}

In this section, we develop an efficient numerical scheme, termed as the Galerkin-L1-POD scheme, for problem
\eqref{eqn:fde}. It is based on the following three components: standard Galerkin method with continuous
piecewise linear finite elements in space, L1 approximation in time and proper orthogonal decomposition in
the snapshot space, which we shall describe separately in the following three subsections.

\subsection{Space discretization by the Galerkin FEM}
First we describe the spatial discretization based on the Galerkin FEM.
Let $\mathcal{T}_h$ be a shape regular and quasi-uniform triangulation of the
domain $\Omega $ into $d$-simplexes, known as finite elements and denoted by $T$. Then
over the triangulation $\mathcal{T}_h$ we define a continuous piecewise linear finite
element space $X_h$ by
\begin{equation*}
  X_h= \left\{v_h\in H_0^1(\Omega):\ v_h|_T \mbox{ is a linear function},\ \forall T \in \mathcal{T}_h\right\}.
\end{equation*}
On the space $X_h$, we define the $L^2(\Omega)$-orthogonal projection $P_h:L^2(\Omega)\to X_h$ by
$
    (P_h \fy,\chi)  =(\fy,\chi) $ for all $\chi\in X_h.
$
Then the semidiscrete Galerkin scheme
for problem \eqref{eqn:fde} reads: find $u_h(t)\in X_h$ such that
\begin{equation}\label{eqn:fem}
  (\Dal u_h,\chi) + (\nabla u_h,\nabla \chi) = (f,\chi)\quad\forall \chi\in X_h, ~t>0,
\end{equation}
with $u_h(0)=v_h\in X_h$.
Upon introducing the discrete Laplacian $\Delta_h: X_h\to X_h$ defined by
$-(\Delta_h\fy,\chi)=(\nabla\fy,\nabla\chi)$ for all $\fy,\,\chi\in X_h$,
the semidiscrete scheme \eqref{eqn:fem} can be rewritten into
\begin{equation}\label{eqn:fdesemidis}
  \Dal u_h(t) +A_h u_h(t) = f_h(t)\quad t>0,
\end{equation}
with $u_h(0)=v_h\in X_h$, $f_h=P_hf$ and $A_h=-\Delta_h$.

\subsection{Time discretization by L1 scheme}
For the time discretization, we divide the interval $[0,T]$ into $N$ equally spaced subintervals with
a time step size $\tau=T/N$, and $t_n=n\tau$, $n=0,\ldots,N$. Then the L1 scheme \cite{LinXu:2007,SunWu:2006}
approximates the Caputo fractional derivative $\Dal u(x,t_n)$ by
 \begin{equation}\label{eqn:L1approx}
   \begin{aligned}
     \Dal u(x,t_n) &= \frac{1}{\Gamma(1-\al)}\sum^{n-1}_{j=0}\int^{t_{j+1}}_{t_j}
        \frac{\partial u(x,s)}{\partial s} (t_n-s)^{-\al}\, ds \\
     &\approx \frac{1}{\Gamma(1-\al)}\sum^{n-1}_{j=0} \frac{u(x,t_{j+1})-u(x,t_j)}{\tau}\int_{t_j}^{t_{j+1}}(t_n-s)^{-\al}ds\\
     &=\sum_{j=0}^{n-1}b_j\frac{u(x,t_{n-j})-u(x,t_{n-j-1})}{\tau^\alpha\Gamma(2-\al)}
    =:\bPtau^\al u(t_n),
   \end{aligned}
 \end{equation}
where the weights $\{b_j\}$ are given by
\begin{equation}\label{eqn:b}
  b_j=(j+1)^{1-\alpha}-j^{1-\alpha},\quad j=0,1,\ldots,n-1.
\end{equation}

Then the fully discrete scheme reads: given $U_h^0=v_h\in X_h$ and $F_h^n=P_hf(t_n)\in X_h$, with
$c_\al=\Gamma(2-\al)$, find $U_h^n\in X_h$ for $n=1,2,\ldots,N$ such that
\begin{equation}\label{eqn:fully}
    (b_0I+c_\al\tau^\al A_h)U_h^n=  b_{n-1} U_h^0+ \sum_{j=1}^{n-1}(b_{j-1}-b_j)U_h^{n-j}+c_\al\tau^\al F_h^n.
\end{equation}

The computational challenge of the fully discrete scheme \eqref{eqn:fully} is obvious: To compute the numerical
solution $U_h^n$ at $t_n$, the solutions
$\{U_h^k\}_{k=0}^{n-1}$ at all preceding time instances are required, as a result of the nonlocality of
the Caputo fractional derivative $\partial_t^\alpha u$. Hence, the computational complexity
and storage requirement grow linearly as the number $n$ of time steps increases, which poses a significant challenge
especially for high-dimensional problems and multi-query applications. This naturally motivates the development of cheap reduced order models by
the POD methodology so as to reduce the effective degree of freedom.

\subsection{Galerkin-L1-POD scheme}\label{ssec:POD}
Now we develop an efficient Galerkin approximation scheme based on proper orthogonal decomposition (POD)
to circumvent the challenge. We shall first describe the general framework of the POD methodology,
and then discuss its application to the subdiffusion equation.

POD is a powerful model reduction technique for complex models, especially time/parameter dependent partial
differential equations. It resides on the empirical observation that despite the large apparent dimensionality
of the solution space (e.g., the degree of freedom of the finite element approximation),
the solution actually lives on an effectively much lower dimensional (possibly highly nonlinear) manifold. POD
constructs a problem adapted basis for efficiently approximating the manifold using samples
from the manifold, often known as ``snapshots'', which can be either solutions at different time instances,
different parameter values, or samples generated using relevant physical experiments. The POD basis functions are then employed within either
a Galerkin or Petrov-Galerkin framework to generate a reduced-order model.

Now we recall the general framework of POD. Let $X$ be a real Hilbert space endowed with
an inner product $(\cdot,\cdot)_X$ and norm $\| \cdot \|_X$. Now for $N\in \mathbb{N}$, let $\{y_n\}_{n=1}^N\subset X$
be an ensemble of snapshots and at least one of them is assumed to be nonzero. Then we set
$ \mathfrak{U}= \text{span}\{ y_1, y_2, ... , y_N\} \subset X.$
Let $\text{dim}(\mathfrak{U})=r$  and let $\{\psi_j\}_{j=1}^r$ be an orthonormal basis of the snapshot space $\mathfrak{U}$.
Then any element $y_n$ can be written as
\begin{equation*}
  y_n=\sum_{j=1}^r (y_n, \psi_j)_X \psi_j, \quad~~ n=1,2,...,N.
\end{equation*}
POD chooses an orthonormal basis $\{ \psi_j \}_{j=1}^m$ for $1\le m \le r$
to minimize the following ensemble average:
\begin{equation}\label{eqn:pod}
    \min_{\{ \psi_j \}_{j=1}^m} \frac1N \sum_{n=1}^N \| y_n- \sum_{j=1}^m (y_n, \psi_j)_X \psi_j \|_X^2.
\end{equation}
A solution of problem \eqref{eqn:pod} is called a POD-basis of rank $m$. This optimization problem is
related to the correlation matrix $K\in \mathbb{R}^{N\times N}$ corresponding to the
snapshots $\{ y_n \}_{n=1}^N$, which  is defined by
\begin{equation}\label{eqn:matK}
    K_{ij}= N^{-1}(y_j,y_i)_X,\quad i,j=1,\ldots, N.
\end{equation}
By its very construction, the matrix $K$ is symmetric positive semidefinite, and its eigenvectors
can be chosen to be orthonormal (in the inner product $(\cdot,\cdot)_X$). Further, the number of
positive eigenvalues is equal to $r$, the dimensionality of the space $\mathfrak{U}$ spanned by the snapshots
(or equivalently the rank of $K$). The following lemma gives the formula of the
POD-basis and the corresponding approximation error within the ensemble \cite{Sirovich:1987}.

\begin{lemma}\label{key}
Let $\lambda_1 \ge \lambda_2 \ge ...\ge \lambda_r >0$ be the positive eigenvalues of the correlation matrix $K$ and $v_1,...,v_r \in \mathbb{R}^N$ be
the corresponding orthonormal eigenvectors. Then a POD basis of rank $m\le r$ is given by
 \begin{equation*}
  \psi_j = \frac{1}{\sqrt{\lambda_j}} \sum_{n=1}^N (v_j)_n y_n,
 \end{equation*}
where $ (v_j)_n$ denotes the $n$-th component of the eigenvector $v_j$. Moreover, the error is given by
 \begin{equation*}
  \frac1N \sum_{n=1}^N \| y_n- \sum_{j=1}^m (y_n, \psi_j)_X \psi_j\|_X^2 = \sum_{j=m+1}^r  \lambda_j.
 \end{equation*}
\end{lemma}

Following the abstract framework, for the subdiffusion model \eqref{eqn:fde}, we choose $2N+1$ snapshots as
\begin{equation*}
 y_n=U_h^{n-1}, \quad n=1,2,...,N+1,
\end{equation*}
and the fractional difference quotients (FDQs)
\begin{equation*}
   y_n= \bPtau^\al U_h^{n-N-1},\quad  n=N+2,...,2N+1.
\end{equation*}
The inclusion of FDQs $\{\bPtau^\al U_h^n\}$ into the snapshots $\mathfrak{U}$ is to improve the
error estimate below: it allows directly bounding the error due to the POD approximation to the
fractional derivative term $\bPtau^\al U_h^n$, cf. Lemma \ref{key}.
In the absence of these FDQs in the snapshots, the error estimate due to POD approximation would involve an additional factor
$\tau^{-2\alpha}$; see Remark \ref{rmk:DQ} for details. The use of difference quotients was first proposed by Kunisch
and Volkwein \cite{KunischVolkein:2001} for the standard parabolic equation, and we refer
interested readers to the recent work \cite{IliescuWang:2014} for extensive discussions. In this
work, we shall follow the work \cite{KunischVolkein:2001}, and
employ the FDQs $\bPtau^\al U_h^n$ in the construction of the POD basis.

In practice, there are several possible choices of the Hilbert space $X$, and we shall consider two
popular ones in this work. Our first choice for the POD space is $X=H_0^1(\Omega)$
with the inner product $(u,v)_X=(\nabla u,\nabla v)$ for all $u,~~v\in H_0^1(\Omega)$.
Then the correlation matrix $\widetilde K$ is given by
\begin{equation}\label{eqn:Ktilde}
    \widetilde K_{i,j}=(2N+1)^{-1}(\nabla y_j,\nabla y_i).
\end{equation}
We denote the corresponding POD basis (called $H_0^1(\Omega)$ POD basis) by $\{\wpsi_j\}_{j=1}^r$ and the
subspace spanned by the first $m$ $H_0^1(\Omega)$-POD basis functions by $X_h^m$, $m\leq r$. Then Lemma \ref{key}
yields the following error estimate for the POD space $X_h^m$
 \begin{equation}\label{eqn:eh}
  \frac1{2N+1}
  \bigg(\sum_{n=0}^N \| U_h^n- \sum_{j=1}^m (\nabla U_h^n,\nabla \wpsi_j) \wpsi_j \|_{H_0^1\II}^2
  +  \sum_{n=1}^N \| \bPtau^\al  U^n_h- \sum_{j=1}^m (\nabla \bPtau^\al U_h^n, \nabla \wpsi_j) \wpsi_j \|_{H_0^1\II}^2\bigg)
  = \sum_{j=m+1}^r \widetilde \la_j,
 \end{equation}
where $\{\widetilde\la_j\}_{j=1}^r$ are the descendingly ordered eigenvalues of the correlation matrix $\widetilde K$.
The second choice is $X=L^2(\Omega)$ with the standard inner product. The correlation matrix
$\widehat K$ is given by
\begin{equation}\label {eqn:Khat}
  \widehat K _{ij} =(2 N+1)^{-1} ( y_j, y_i).
\end{equation}
Likewise, we denote the corresponding POD basis (called $L^2\II$-POD basis) by $\{\widehat \psi\}_{j=1}^r$,
and by slightly abusing the notation, the subspace spanned by the first $m$ $L^2\II$ POD basis functions by $X_h^m$.
Then in view of Lemma \ref{key}, the POD space $X_h^m$ satisfies the following error estimate
\begin{equation}\label{eqn:el}
  \frac1{2N+1}
  \bigg(\sum_{n=0}^N \| U_h^n - \sum_{j=1}^m (U_h^n, \hpsi_j) \hpsi_j\|_{L^2\II}^2
  +\sum_{n=1}^N \|\bPtau^\al U_h^n - \sum_{j=1}^m (\bPtau^\al U_h^n, \hpsi_j) \hpsi_j\|_{L^2\II}^2\bigg)
  = \sum_{j=m+1}^r  \widehat \la_j,
\end{equation}
where $\{\widehat\la_j\}_{j=1}^r$ are the descendingly order eigenvalues of the correlation matrix $\widehat K$.

Next we define the Ritz projection operator $R_h^m: X_h \to X_h^m$ by
\begin{equation}\label{eqn:Ritz-pod}
 (\nabla R_h^m \chi, \nabla \fy) = (\nabla \chi, \nabla \fy)\quad \forall\fy\in X_h^m,
\end{equation}
where $\chi\in X_h\subset H_0^1(\Omega)$. The $H^1\II$-stability of the projection operator $R_h^m$
on the space $X_h$ is immediate
\begin{equation*}
  \|  \nabla R_h^m \chi  \|_{L^2\II}\le \| \nabla \chi  \|_{L^2\II} \quad \forall \chi\in X_h.
\end{equation*}

Given the POD basis, one can exploit it for model reduction in several different ways. One natural choice is to
use a Galerkin approach, which yields the following reduced-order formulation: with $U_m^0=R_h^m v_h\in X_h^m$,
 find $U_m^n \in X_h^m$, $n=1,2,...,N$ such that
\begin{equation}\label{eqn:fully-reduced}
  (\bar\partial_\tau^\alpha U_m^n,\fy_m) +(\nabla U_m^n,\nabla\fy_m) = (f(t_n),\fy_m) \quad \forall \fy_m\in X_h^m,
\end{equation}
or equivalently with $c_\al=\Gamma(2-\al)$,
\begin{equation*}
    b_0(U_m^n, \fy_m) + c_\al\tau^\al (\nabla U_m^n, \nabla\fy_m)
    =  b_{n-1} (U_m^0, \fy_m)+ \sum_{j=1}^{n-1}(b_{j-1}-b_j)(U_m^{n-j}, \fy_m)+c_\al\tau^\al (f(t_n),\fy_m) \quad\forall \fy_m\in X_h^m.
\end{equation*}
The existence and uniqueness of the POD approximation $\{U_m^n\}_{n=1}^N$ follows directly by an energy argument (see Section
\ref{sec:conv} below). In the Galerkin
framework, the stiffness matrix of the reduced-order formulation \eqref{eqn:fully-reduced} is the projection of
that of the global one \eqref{eqn:fully} into the POD space $X_h^m$. It is worth mentioning that the degree of
freedom of the reduced system \eqref{eqn:fully-reduced} is $m$, i.e., the number of POD basis functions in $X_h^m$,
which is usually much smaller than that of \eqref{eqn:fully}, i.e., the number of finite element basis functions. This
shows clearly the enormous gain in the computational complexity and storage requirement of the proposed scheme.

\section{Error analysis} \label{sec:conv}

In this part, we provide a complete error analysis of the proposed scheme \eqref{eqn:fully-reduced}.
The discretization error consists of three sources: the spatial discretization,
temporal discretization and POD approximation. It is known that the semidiscrete solution
$u_h$ satisfies the following nearly optimal error estimate \cite{JinLazarovZhou:2013,JinLazarovPasciakZhou:2013a},
where the operator $A$ is the negative Laplacian operator $-\Delta$ with a zero Dirichlet boundary condition.
The log factor $\ell_h^2$ in the error estimate is due to the limited smoothing property of
the solution operator for subdiffusion, and the prefactor $t^{-\alpha(1-\sigma)}$, for
$t\to0$, reflects the corresponding solution singularity.
\begin{theorem}\label{thm:error-semi}
Let $u$ be the solution of problem \eqref{eqn:fde} with $A^\sigma v\in L^2\II$, $0<\sigma\leq 1$, and
$f\in L^\infty(0,T;L^2\II)$, and $u_h$ be the solution of problem \eqref{eqn:fdesemidis} with $v_h=P_hv$
and $f_h=P_hf$. Then there holds with $\ell_h=|\log h| $
\begin{equation*}
  \|u(t)-u_h(t)\|_{L^2(\Omega)}\leq ch^2\ell_h^2 \left(t^{-\alpha(1-\sigma)}\|A^\sigma v\|_{L^2\II}+\|f\|_{L^\infty(0,T;L^2(\Omega))}\right).
\end{equation*}
\end{theorem}

Below we derive the errors due to the temporal approximation and the POD approximation that are expressed
in terms of the data regularity directly. The main novel ingredient in the convergence analysis is to
establish a suitable stability result for the L1 time stepping under realistic assumptions on the
data regularity. To this end, we shall develop a novel energy argument, based on the monotonicity of
a suitable quadrature rule.

\subsection{Error analysis of the L1 scheme}
Now we develop a novel energy argument for analyzing the L1 approximation.
We begin with a weighted inequality for the weights $\{b_j\}$, which is crucial for
establishing the monotonicity of the quadrature below.
\begin{lemma}\label{lem:mid-term}
Let $\{b_j\}$ be defined by \eqref{eqn:b}. Then for $j=2,\ldots,n-1$, there holds
\begin{equation*}
   (j-1)n^{\alpha-2}b_{j-1} + (n-j)n^{\alpha-2}b_j\leq (n+1)^{\alpha-1}b_j.
\end{equation*}
\end{lemma}
\begin{proof}
Using the definition of the weights $b_j$, the assertion is equivalent to: for all $j=2,\ldots,n-1$:
\begin{equation*}
  \int_0^1 (j-1+t)^{-\alpha}(j-1) - \left(n\left(1+n^{-1}\right)^{\alpha-1}-n+j\right)(j+t)^{-\alpha}dt \leq 0,
\end{equation*}
that is,
\begin{equation*}
  \int_0^1\frac{g(t)}{(j-1+t)^\alpha (j+t)^\alpha}dt \leq 0,
\end{equation*}
where the function $g:[0,1]\to\mathbb{R}$ is defined by
$
  g(t) = (j-1) (j+t)^{\alpha} - (j-1+t)^\alpha(n(1+n^{-1})^{\alpha-1}-n+j),
$
with its $g'(t)$ given by
\begin{equation*}
  g'(t) =\alpha\left[ \frac{j-1}{(j+t)^{1-\alpha}} - \frac{n(1+n^{-1})^{\alpha-1}-n+j}{(j-1+t)^{1-\alpha}}\right].
\end{equation*}
For $\alpha\in(0,1)$, there holds $n(1+n^{-1})^{\alpha-1} -n + j \geq n^2(n+1)^{-1}-n+j = j - n(n+1)^{-1}>j-1$. Hence
we deduce $g'(t)<0$ on the interval $[0,1]$. It suffices to show that $g(0)\leq 0$. Obviously,
\begin{equation*}
    g(0)= (j-1)^\alpha(\underbrace{(j-1)^{1-\alpha}j^\alpha - j + n\left(1-(1+n^{-1})^{\alpha-1}\right)}_{\mathrm{I}}).
\end{equation*}
The term $\mathrm{I}$ in the bracket can be rewritten as
\begin{equation*}
  \mathrm{I} = j\left((1-j^{-1})^{1-\alpha}-1\right) + n\left(1-(1+n^{-1})^{\alpha-1}\right).
\end{equation*}
We claim that the function $\tilde g(j)= j \left(1-(1-j^{-1})^{1-\alpha}\right)$ is monotonically decreasing in $j$. To see this,
let $h(t): (0,1)\to \mathbb{R}$, with $h(t) = t^{-1} (1-(1-t)^{1-\alpha})$. Then
$
    h'(t)  
    =-t^{-2} \left(1-(1-t)^{-\alpha} (1-\al t)\right).
$
Next consider the function $\tilde h(t): (0,1)\to\mathbb{R}$, with $\tilde{h}(t)=(1-t)^{\alpha}$. Then $\tilde{h}'(t)=-\al(1-t)^{\alpha-1}$ and $\tilde{h}''(t)=(\alpha-1)\alpha(1-t)^{\alpha-2}<0$, namely,
the function $\tilde{h}$ is concave. Then the concavity implies $\tilde{h}(t)\leq \tilde{h}(0) + \tilde{h}'(0)t$,
which gives $(1-t)^{\alpha}\leq 1 - \alpha t$. Consequently,
$h'(t) \ge - t^{-2} (1-(1-\al t)^{-1} (1-\al t) ) \geq 0,$ and hence $h$ is
monotonically increasing, and the monotonicity of the function $\tilde{g}(j)$ follows.
Hence, by the trivial inequality $(n-1)/n<n/(n+1)$, we have
\begin{equation*}
   \begin{aligned}
     \mathrm{I} & < n((1-n^{-1})^{1-\alpha}-1) + n(1-(1+n^{-1})^{\alpha-1})\\
     & = n \big((1-n^{-1})^{1-\alpha}-(1-(n+1)^{-1})^{1-\alpha}\big)<0,
   \end{aligned}
\end{equation*}
which concludes the proof of the lemma.
\end{proof}

Now we give an important monotonicity relation of a weighted rectangular quadrature approximation.
\begin{theorem}\label{lem:convex-quad}
Let the function $f: [0,1]\to\mathbb{R}$ be convex and nonnegative with $f(0)=0$, and $\alpha\in(0,1)$. For
any $n\in\mathbb{N}$, let $x_j=\frac{j}{n}$, $j=0,\ldots,n$, and $y_j=\frac{j}{n+1}$, $j=0,\ldots,n+1$. Then
there holds
\begin{equation*}
  n^{\alpha-1}\sum_{j=0}^{n-1}b_jf(x_j) \leq (n+1)^{\alpha-1}\sum_{j=0}^nb_jf(y_j).
\end{equation*}
\end{theorem}
\begin{proof}
First we observe the trivial inequalities $\frac{j}{n+1}<\frac{j}{n}<\frac{j+1}{n+1}$, i.e.,
$y_j<x_j<y_{j+1}$, for $j=1,\ldots,n-1$. There also holds the trivial identity
\begin{equation*}
  x_j:= \frac{j}{n} = \frac{n-j}{n}\frac{j}{n+1} + \frac{j}{n}\frac{j+1}{n+1} =:\frac{n-j}{n}y_j + \frac{j}{n}y_{j+1}.
\end{equation*}
Now by the convexity of the function $f$, we deduce
\begin{equation*}
  f(x_j) = f\left(\frac{n-j}{n} y_j + \frac{j}{n}y_{j+1}\right)\leq \frac{n-j}{n}f(y_j) + \frac{j}{n}f(y_{j+1}),\quad j=1,\ldots,n.
\end{equation*}
With the assumption $f(0)=0$, it suffices to consider $j\geq1$ in the sum. Hence
\begin{equation*}
  \begin{aligned}
  n^{\alpha-1}\sum_{j=1}^{n-1}b_jf(x_j) & \leq n^{\alpha-1} \sum_{j=1}^{n-1}b_j\left(\frac{n-j}{n}f(y_j) + \frac{j}{n}f(y_{j+1})\right)\\
   &= n^{\alpha-1}\left(b_1\frac{n-1}{n}f(y_1) +
   \sum_{j=2}^{n-1}\left(b_{j-1}\frac{j-1}{n}+\frac{n-j}{n}b_{j}\right)f(y_{j}) +b_{n-1}\frac{n-1}{n}f(y_n)\right).
  \end{aligned}
\end{equation*}
To show the desired assertion, we consider the following three cases separately, first, last and middle terms.
For the first term, in view  of the nonnegativity of the function $f$, it suffices to show
$n^{\alpha-1}\frac{n-1}{n}b_1 \leq (n+1)^{\alpha-1} b_1$, which however follows from $\alpha\in(0,1)$ and
\begin{equation*}
  (n+1)^{1-\alpha} n^{\alpha-1}\frac{n-1}{n} = \left(\frac{n+1}{n}\right)^{1-\alpha} \frac{n-1}{n}
   = \left(\frac{n^2-1}{n^2}\right)^{1-\alpha}\left(\frac{n-1}{n}\right)^\alpha< 1.
\end{equation*}
For the last term, we have
\begin{equation*}
  n^{\alpha-1} \frac{n-1}{n}b_{n-1} = n^{\alpha-1}(n^{1-\alpha}-(n-1)^{1-\alpha})\frac{n-1}{n} = 1 - \frac{1}{n} - \left(\frac{n-1}{n}\right)^{2-\alpha},
\end{equation*}
and meanwhile
\begin{equation*}
  (n+1)^{\alpha-1}b_n = (n+1)^{\alpha-1}((n+1)^{1-\alpha}-n^{1-\alpha})= 1 -\left(\frac{n}{n+1}\right)^{1-\alpha}.
\end{equation*}
Hence, it suffices to show $n^{-1}+(1-n^{-1})^{2-\alpha}-(1-(n+1)^{-1})^{1-\alpha}>0$ for $n>1$.
Let $g:[0,1]\to\mathbb{R}$ by $g(t)=t+(1-t)^{2-\alpha}-(1+t)^{\alpha-1}$. Then $g(0)=0$, and
$ g'(t)  = 1 - (2-\alpha)(1-t)^{1-\alpha} - (\alpha-1)(1+t)^{\alpha-2}$.
Clearly $g'(0)=0$ and further
$g''(t) = (2-\alpha)(1-\alpha)((1-t)^{-\alpha}-(1+t)^{\alpha-3})>0,$
which in particular implies $g'(t)\geq0 $ on the interval $[0,1]$.
To conclude the proof, it suffices to show the inequality for the middle terms, i.e., for $j=2,\ldots,n-1$
\begin{equation*}
  n^{\alpha-1}\frac{j-1}{n}b_{j-1} + n^{\alpha-1}\frac{n-j}{n}b_j\leq (n+1)^{\alpha-1}b_j,
\end{equation*}
which however is already shown in Lemma \ref{lem:mid-term}.
\end{proof}

The following result is a direct corollary from Theorem \ref{lem:convex-quad}, and
it will play a crucial role in establishing the stability result in Theorem
\ref{thm:stab-uni} below.
\begin{lemma}\label{lem:key1}
For any $\alpha\in(0,1)$, let $b_j$ be defined in \eqref{eqn:b}. Then for any $n\in \mathbb{N}$, there holds
\begin{equation*}
  \sum_{j=1}^n(b_{j-1}-b_j)(n+1-j)^{\alpha-1} \leq (n+1)^{\alpha-1}.
\end{equation*}
\end{lemma}
\begin{proof}
Consider the function $f(x)=(1-x)^{\alpha-1}-1$. Then it satisfies $f(x)\geq 0$, $f(0)=0$, and
also $f''(x)> 0$, i.e., convex. Hence, by Theorem \ref{lem:convex-quad}, we have
\begin{equation}\label{eqn:approx-conv}
  n^{\alpha-1}\sum_{j=0}^{n-1} b_j\big((1-jn^{-1})^{\alpha-1}-1\big) \leq (n+1)^{\alpha-1}\sum_{j=1}^n b_j\big((1-{j}({n+1})^{-1})^{\alpha-1}-1\big).
\end{equation}
Meanwhile, it can be verified directly that for all $n\in\mathbb{N}^+$,
$
  n^{\alpha-1} \sum_{j=0}^{n-1}b_j = (1-\al)\int_0^1 x^{-\al}\,dx=1,
$
i.e.,
$
  n^{\alpha-1}\sum_{j=0}^{n-1}b_j = (n+1)^{\alpha-1} \sum_{j=0}^nb_j.
$
Plugging the preceding identity into \eqref{eqn:approx-conv} yields
\begin{equation*}
  n^{\alpha-1}\sum_{j=0}^{n-1} b_j\left(1-{j}{n}^{-1}\right)^{\alpha-1} \leq (n+1)^{\alpha-1}\sum_{j=0}^n b_j\left(1-{j}({n+1})^{-1}\right)^{\alpha-1}.
\end{equation*}
which upon rearranging terms gives the desired assertion.
\end{proof}

Next we give an important $L^2(\Omega)$ stability result. The stability estimate
puts more weights on the source term $F_h^{k}$ as the index $k$ gets close to the
current time step $n$, in a manner analogous to the continuous problem.
\begin{theorem}\label{thm:stab-uni}
 Let $U_h^n$, $n=1,2,\ldots,N$, be the solution of the fully discrete scheme \eqref{eqn:fully}. Then
 with $c_\alpha=\Gamma(2-\alpha)$, for $n=1,2,\ldots,N$, we have the following stability estimate
\begin{equation}\label{eqn:stab-uni}
 \|  U_h^n \|_{L^2\II} \le \| v_h  \|_{L^2\II} + c_\al\tau^\al \sum_{k=0}^{n-1} (n-k)^{\al-1} \|  F_h^{k+1} \|_{L^2\II}.
\end{equation}
\end{theorem}
\begin{proof}
We show the assertion by mathematical induction. First
we consider the case $n=1$. Multiplying both sides of \eqref{eqn:fully} by $U_h^1$
and integrating over the domain $\Om$ yield
\begin{equation*}
 \|  U_h^1 \|_{L^2\II}^2+c_\al\tau^\al \|\nabla U_h^1\|_{L^2\II}^2= (U_h^0,U_h^1)+ c_\al\tau^\al (F_h^1,U_h^1).
\end{equation*}
Then the Cauchy-Schwartz inequality and Young's inequality give
\begin{equation*}
  \|  U_h^1\|_{L^2\II} \le  \|  U_h^0 \|_{L^2\II}+ c_\al\tau^\al \|F_h^1\|_{L^2\II}.
\end{equation*}
Now assume the estimate holds up to some $n\geq 1$. A similar argument yields
\begin{equation*}
\begin{split}
   \| U_h^{n+1}\|_{L^2(\Omega)} &\le b_n \|  U_h^0 \|_{L^2(\Omega)} + \sum_{j=1}^{n} (b_{j-1}-b_{j})  \| U^{n+1-j} \|_{L^2(\Omega)}
    +  c_\al \tau^\al \|  F^{n+1}  \|_{L^2(\Omega)} \\
    & \leq b_n \|  U_h^0 \|_{L^2(\Omega)} + \sum_{j=1}^{n} (b_{j-1}-b_{j}) \bigg(\|U_h^0\|_{L^2\II}+  c_\al \tau^\al \|  F^{n+1}  \|_{L^2(\Omega)} \\
    &\quad +c_\alpha\tau^\alpha\sum_{k=0}^{n-j}(n+1-j-k)^{\alpha-1}\|F_h^{k+1}\|_{L^2\II}\bigg)\\
    &= \|  U_h^0 \|_{L^2(\Omega)}  +
    c_\al\tau^\al \sum_{j=1}^{n} (b_{j-1}-b_{j})\sum_{k=0}^{n-j} (n+1-j-k)^{\al-1} \|  F_h^{k+1} \|_{L^2\II}
    +c_\al \tau^\al \|  F^{n+1}  \|_{L^2(\Omega)} .
\end{split}
\end{equation*}
Then by changing the order of summation and applying Lemma \ref{lem:key1} we have
\begin{equation*}
\begin{split}
 \sum_{j=1}^{n} (b_{j-1}-b_{j})\sum_{k=0}^{n-j} (n+1-j-k)^{\al-1} \|  F_h^{k+1} \|_{L^2\II}
=&\sum_{k=0}^{n-1}  \|  F_h^{k+1} \|_{L^2\II} \sum_{j=1}^{n-k} (b_{j-1}-b_{j})(n+1-j-k)^{\al-1} \\
\le& \sum_{k=0}^{n-1}  \|  F_h^{k+1} \|_{L^2\II} (n+1-k)^{\al-1}, \\
\end{split}
\end{equation*}
and consequently
\begin{equation*}
\begin{split}
   \| U_h^{n+1}\|_{L^2(\Omega)}
    &\le \|  U_h^0 \|_{L^2(\Omega)}  +
    c_\al\tau^\al \sum_{k=0}^{n-1} (n+1-k)^{\al-1} \|  F_h^{k+1} \|_{L^2\II}
    +c_\al \tau^\al \|  F^{n+1}  \|_{L^2(\Omega)} \\
    &= \|  U_h^0 \|_{L^2(\Omega)} + c_\al\tau^\al \sum_{k=0}^{n}(n+1-k)^{\al-1}  \|  F_h^{k+1} \|_{L^2\II},
\end{split}
\end{equation*}
which completes the induction step and the desired assertion follows.
\end{proof}

The next lemma gives one useful estimate for bounding the local truncation error.
\begin{lemma}\label{lem:key2}
For any $\delta\in(0,\al]$, there exists a constant $c>0$, independent of $n$, such that for all $n\geq 2$
\begin{equation*}
\sum_{k=1}^{n-1} \left((n-k)^{1-\al}-(n-k-1)^{1-\al} \right)
k^{\delta-2} \le c (n-1)^{-\al}.
\end{equation*}
\end{lemma}
\begin{proof}
The case $n=2$ is trivial, and we consider only $n\ge 3$.
Let $d_k=((n-k)^{1-\al}-(n-k-1)^{1-\al} ) k^{\delta-2}$.
First, we observe that for $k=1$
\begin{equation*}
  \begin{aligned}
    d_1  &= (n-1)^{1-\al}-(n-2)^{1-\al}= (1-\alpha)\int_{1}^{2}(n-s)^{-\al}\,ds\\
     &{\le} c (n-2)^{-\al}\le c((n-1)/3)^{-\al}\le c(n-1)^{-\al}.
  \end{aligned}
\end{equation*}
The sum of the remaining terms can be bounded directly by
\begin{equation*}
\begin{split}
  \sum_{k=2}^{n-1}d_k & = \sum_{k=2}^{n-1} (1-\al)k^{\delta-2}\int_k^{k+1} (n-s)^{-\al}\,ds
  \le c\sum_{k=2}^{n-1} \int_k^{k+1} (n-s)^{-\al}(s-1)^{\delta-2}\,ds \\
    &=  c \int_2^{n} (n-s)^{-\al}(s-1)^{\delta-2}\,ds = c \int_1^{n-1} (n-s-1)^{-\al}s^{\delta-2}\,ds\\
    &= c \int_1^{\frac{n-1}{2}}(n-s-1)^{-\alpha}s^{\delta-2}ds + c
  \int_{\frac{n-1}{2}}^{n-1} (n-s-1)^{-\al}s^{\delta-2}\,ds := \mathrm{I} + \mathrm{II}.
\end{split}
\end{equation*}
Then the desired result follows from
\begin{equation*}
  \mathrm{I} \le c\int_1^{\frac{n-1}{2}} (n-s-1)^{-\al}s^{\delta-2}\,ds \le c(n-1)^{-\al}\int_1^{\frac{n-1}{2}} s^{\delta-2}\,ds\le c(n-1)^{-\al}
\end{equation*}
and
\begin{equation*}
  \mathrm{II}\le c (n-1)^{\delta-2}\int_{\frac{n-1}{2}}^{n-1} (n-s-1)^{-\al}\,ds \le c (n-1)^{\delta-\al-1} \le c(n-1)^{-\al}.
\end{equation*}
\end{proof}

Next we derive an error bound on the local truncation error $r_n$ defined by
\begin{equation}\label{eqn:trun}
    r_n = \|\partial_t^\alpha u_h(t_n)-\bPtau^\al u_h(t_n)\|_{L^2\II},\quad \quad n=1,2,...,N.
\end{equation}
In view of Theorems \ref{thm:reg-init} and \ref{thm:reg} in the appendix, we make the following temporal regularity assumption.
\begin{assumption}\label{ass:reg}
The solution $u$ satisfies the following smoothing properties
$$ \|  u(t) \|_{L^2\II} \le c\quad \text{and}   \quad \| \partial_t^m u(t)  \|_{L^2\II} \le ct^{\delta-m},  $$
where $\delta>0$ and the integer $m\ge1$.
\end{assumption}

\begin{remark}
By Theorems \ref{thm:reg-init} and \ref{thm:reg}, the regularity condition in Assumption \ref{ass:reg} holds with
$\delta=\sigma\alpha$, $\sigma\in(0,1]$, for initial data $v\in D(A^\sigma)$ and source term
$f\in W^{2,\infty}(0,T;L^2(\Omega))$. Under these conditions, Assumption \ref{ass:reg} holds also for the
semidiscrete Galerkin approximation $u_h$, with a constant $c$ independent of $h$.
\end{remark}

\begin{lemma}\label{thm:trun}
Let Assumption \ref{ass:reg} hold, and $r_n$ be the local truncation error defined by \eqref{eqn:trun}. Then
\begin{equation*}
 r_n\leq
\begin{cases}
\quad c\tau^{\delta-\al} \quad & \text{if } n=1,  \\
\quad c(n-1)^{-\al}\tau^{\delta-\al}\quad &\text{if } n\ge2.
\end{cases}
\end{equation*}
\end{lemma}
\begin{proof}
Using Assumption \ref{ass:reg}, for $n=1$, we have the following estimate (with $c_\alpha^\prime=1/\Gamma(2-\alpha)$)
\begin{equation}\label{eqn:r1}
  \begin{split}
    r_1  
    &\le c^\prime_\alpha\tau^{-1}
     \bigg|\hspace{-0.6mm}\bigg|\int_0^{\tau} (\tau-s)^{-\al} \int_0^{\tau} (u_h'(s)-u_h'(y))\,dy\,ds  \bigg|\hspace{-0.6mm}\bigg|_{L^2\II}\\
   &\le c^\prime_\alpha\tau^{-1}\int_0^{\tau} (\tau-s)^{-\al} \int_0^{\tau} \|u_h'(s)\|_{L^2\II}  + \|u_h'(y)\|_{L^2\II}\,dy\,ds  \\
    &\le c^\prime_\alpha\tau^{-1}\int_0^{\tau} (\tau-s)^{-\al} \int_0^{\tau} (s^{\delta-1}  + y^{\delta-1})\,dy\,ds \le c \tau^{\delta-\al}.
  \end{split}
\end{equation}
Now we consider the case $n\ge2$. Then
\begin{equation*}
\begin{split}
 r_n  &= c_\alpha^\prime\|
 \sum_{k=0}^{n-1}\int_{t_k}^{t_{k+1}} (t_n-s)^{-\al} \bigg(u_h'(s) - \frac{u_h(t_{k+1})-u_h(k)}{\tau}\bigg)\,ds  \|_{L^2\II}\\
 &\le c\sum_{k=0}^{n-1} \|\int_{t_k}^{t_{k+1}} (t_n-s)^{-\al} \bigg(u_h'(s) - \frac{u_h(t_{k+1})-u_h(k)}{\tau}\bigg)\,ds\|_{L^2\II}
 :=  c\sum_{k=0}^{n-1} r_{n,k}.
\end{split}
\end{equation*}
The first term $r_{n,0}$ can be bounded using Assumption \ref{ass:reg} and the argument for \eqref{eqn:r1} as
\begin{equation}\label{eqn:rn0}
\begin{split}
r_{n,0}  
&\le  c\int_{0}^{t_{1}} (t_n-s)^{-\al} \|u_h'(s)\|_{L^2\II} \,ds + c\tau^{-1}\int_{0}^{t_{1}} (t_n-s)^{-\al} \int_{0}^{t_{1}} \|u_h'(y)\|_{L^2\II} \,dy \,ds\\
&\le c (t_n-t_1)^{-\al}\int_{0}^{t_{1}} s^{\delta-1}\,ds + {c\tau^{\delta-1}\int_0^{t_1}(t_n-s)^{-\al}\,ds} \le c(n-1)^{-\al} \tau^{\delta-\al}.
\end{split}
\end{equation}
Next we derive estimates for $r_{n,k}$, $k=1,2,...,n-1$. To this end,
we use the identity
\begin{equation*}
u_h'(s) - \frac{u_h(t_{k+1})-u_h(k)}{\tau}
 = \frac{1}{\tau}\int_{t_k}^{t_{k+1}} u_h'(s)-u_{h}'(y)\,dy
  = \frac{1}{\tau}\int_{t_k}^{t_{k+1}} \int_{y}^{s}u_h''(z)\,dz\,dy
\end{equation*}
and apply Assumption \ref{ass:reg} such that $ \|   u_{h}''(z) \|_{L^2\II} \le c t_k^{\delta-2}$ with  $c$ independent of $t$ and $h$
to deduce
\begin{equation*}
\bigg|\hspace{-0.6mm}\bigg|u_h'(s) - \frac{u_h(t_{k+1})-u_h(k)}{\tau}\bigg|\hspace{-0.6mm}\bigg|_{L^2\II}
\le \frac{1}{\tau}\int_{t_k}^{t_{k+1}} \int_{\min(s,y)}^{\max(s,y)}  \| u_h''(z)\|_{L^2\II}\,dz\,dy \le c \tau t_k^{\delta-2}.
\end{equation*}
Thus we obtain
\begin{equation*}
\begin{split}
r_{n,k} &\le c\tau t_k^{\delta-2} \int_{t_k}^{t_{k+1}} (t_n-s)^{-\al}\,ds
= c\tau^{2-\al} t_k^{\delta-2}\left( (n-k)^{1-\al}-(n-k-1)^{1-\al}\right)\\
&=c\tau^{\delta-\al} k^{\delta-2} \left( (n-k)^{1-\al}-(n-k-1)^{1-\al}\right).
\end{split}
\end{equation*}
Then by Lemma \ref{lem:key2} we deduce
\begin{equation*}
\sum_{k=1}^{n-1} r_{n,k} \le c\tau^{\delta-\al} \sum_{k=1}^{n-1} k^{\delta-2} \left( (n-k)^{1-\al}-(n-k-1)^{1-\al}\right)
\le c\tau^{\delta-\al} (n-1)^{-\al}.
\end{equation*}
This together with \eqref{eqn:rn0} yields the desired estimate and hence completes the proof.
\end{proof}

Next we derive the error estimate $e_h^n=u_h(t_n)-U_h^n$, $n=1,2,...,N$. First, we observe that
the nodal error $e_h^n$ satisfies $e_h^0=0$ and the following error equation
\begin{equation*}
 \bar{\partial}_t^\alpha e_h^n + A_h e_h^n = \bar{\partial}_t^\alpha u_h(t_n)-\Dal u_h(t_n).
\end{equation*}

The next theorem gives an optimal (uniform in time $t$) error estimate for the fully discrete scheme \eqref{eqn:fully}.
\begin{theorem}\label{thm:error-nonsmooth}
Assume $f\in W^{2,\infty}(0,T;L^2(\Omega))$ and $v\in D(A^\sigma)$, with $0<\sigma\leq1$.
Let $u_h$ and $U_h^n$ be the solutions of problems \eqref{eqn:fdesemidis} and \eqref{eqn:fully}, respectively. Then there holds
\begin{equation*}
   \| u_h(t_n)-U_h^n \|_{L^2(\Om)} \le c \tau^{\sigma\alpha} \left(\| A^\sigma v \|_{L^2\II} + \| f \|_{W^{2,\infty}(0,T;L^2\II)}\right).
\end{equation*}
\end{theorem}
\begin{proof}
By Theorem \ref{thm:stab-uni} and Lemma \ref{thm:trun}, with $\delta=\sigma\alpha$, we have
\begin{equation*}
 \begin{split}
 \| u_h(t_n)-U_h^n \|_{L^2(\Om)} &\le c\tau^{\al} \sum_{k=0}^{n-1} (n-k)^{\al-1} \|  \bPtau^\al u_h(t_{k+1})-\Dal u_h(t_{k+1}) \|_{L^2\II}\\
 &\le c\tau^{\sigma\alpha}\left(\| A^\sigma v \|_{L^2\II} + \| f \|_{W^{2,\infty}(0,T;L^2\II)}\right)
 \bigg(1+\sum_{k=1}^{n-1} (n-k)^{\al-1} k^{-\al}\bigg).\\
 \end{split}
\end{equation*}
Then the following uniform bound
\begin{equation*}
 \sum_{k=1}^{n-1} (n-k)^{\al-1} k^{-\al} =\frac{1}{n}\sum_{k=1}^{n-1}\left(1-\frac{k}{n}\right)^{\alpha-1}\left(\frac{k}{n}\right)^{-\alpha}
 \leq \int_0^1 (1-x)^{\al-1} x^{-\al}dx \le c
\end{equation*}
yields the desired estimate.
\end{proof}

Last, we can state an error estimate on the fully discrete approximation $U_h^n$, which follows from
Theorems \ref{thm:error-semi} and \ref{thm:error-nonsmooth} by the triangle inequality.
\begin{theorem}\label{thm:error-fully}
Assume $f\in W^{2,\infty}(0,T;L^2\II)$ and $v\in D(A^\sigma)$, with $0<\sigma\leq 1$. Let
$u$ and $U_h^n$ be the solutions of problems \eqref{eqn:fde} and \eqref{eqn:fully}, respectively. Then
with $\ell_h=|\log h|$, there holds
\begin{equation*}
   \| u(t_n)-U_h^n \|_{L^2(\Om)} \le c(h^2\ell_ht^{-\alpha(1-\sigma)}+\tau^{\sigma\alpha})\| A^\sigma v \|_{L^2\II} + c(h^2\ell_h^2+\tau^{\sigma\alpha}) \| f \|_{W^{2,\infty}(0,T;L^2\II)}.
\end{equation*}
\end{theorem}

\subsection{Error analysis of the POD approximation}
Next we derive the error estimates for the POD approximation $U_m^n$.
First we recall an approximation property of the Ritz projection
operator $R_h^m$ defined in \eqref{eqn:Ritz-pod} within the ensemble \cite[Lemma 3 and Corrolary 3]{KunischVolkein:2001}.

\begin{lemma}\label{lem:Rhm}
For every $m=1,...,r$, the Ritz projection operator $R_h^m$ satisfies
\begin{equation*}
\frac{1}{N} \sum_{n=1}^N \left(\|\nabla(U_h^n-R_h^m U_h^n)   \|_{L^2\II}^2 + \| \nabla( \bPtau^\al  U_h^n-\bPtau^\al R_h^m U_h^n)   \|_{L^2\II}^2 \right)
\le c\sum_{j=m+1}^r \widetilde \la_j
\end{equation*}
and
\begin{equation*}
\frac{1}{N} \sum_{n=1}^N \left(\|\nabla(U_h^n-R_h^m U_h^n)   \|_{L^2\II}^2 + \| \nabla( \bPtau^\al  U_h^n-\bPtau^\al R_h^m U_h^n)   \|_{L^2\II}^2 \right)
\le c h^{-2}\sum_{j=m+1}^r  \widehat \la_j
\end{equation*}
where $ \{\widetilde \la_j\}_{j=1}^r $ and $\{ \widehat \la_j\}_{j=1}^r $ denote the eigenvalues of
$\widetilde K$ and $\widehat K$ defined in \eqref{eqn:Ktilde}  and \eqref{eqn:Khat}, respectively.
\end{lemma}

Now we can give the error estimate for the POD approximation $U_m^n$ for smooth problem data. The result indicates that
the error incurred by using the POD basis in place of the full Galerkin FEM basis is determined by
the eigenvalues corresponding to the eigenfunctions that are not included in constructing the POD approximation.
In particular, if the eigenvalues of the correlation matrix decay rapidly, then a small number of POD basis functions in
the Galerkin POD scheme \eqref{eqn:fully-reduced} suffice the desired accuracy.
\begin{theorem}\label{thm:err-pod}
Let $u$ and $U_m^n$ be the solutions of \eqref{eqn:fde} and \eqref{eqn:fully-reduced}, respectively, and
suppose that $v\in D(A)$, and $f\in W^{2,\infty}(0,T; L^2\II)$. Then there holds
\begin{equation}\label{eqn:err-pod-1}
\frac{1}{N} \sum_{n=1}^N \| u(t_n)-U_m^n \|_{L^2\II}^2
\le c_T\bigg(\tau^{2\al} + h^4\ell_h^4  + \sum_{j=m+1}^r \widetilde \la_j\bigg)
\end{equation}
and
\begin{equation}\label{eqn:err-pod-2}
\frac{1}{N} \sum_{n=1}^N \| u(t_n)-U_m^n \|_{L^2\II}^2
\le c_T\bigg(\tau^{2\al} + h^4\ell_h^4   + h^{-2} \sum_{j=m+1}^r \widehat \la_j \bigg),
\end{equation}
where $ \{\widetilde \la_j\}_{j=1}^r $ and $ \{\widehat \la_j\}_{j=1}^r  $ denote the eigenvalues of
$\widetilde K$  and $\widehat K$ defined in \eqref{eqn:Ktilde} and \eqref{eqn:Khat}, respectively.
\end{theorem}
\begin{proof}
We split the error $e_m^n = u(t_n) - U_m^n$ into
\begin{equation*}
 e_m^n = \left(u(t_n) - U_h^n\right) + \left(U_h^n - U_m^n\right),
\end{equation*}
and the first term can be bounded using Theorem \ref{thm:error-fully}, i.e.,
\begin{equation*}
  \frac{1}{N} \sum_{n=1}^N \| u(t_n)-U_h^n \|_{L^2\II}^2 \le
  c\left(\tau^{2\al} + h^4\ell_h^4 \right).
\end{equation*}
Hence it suffices to establish a bound for the second term $U_h^n-U_m^n$. Now we consider the splitting
\begin{equation*}
  U_h^n - U_m^n = \left(U_h^n - R_h^m U_h^n \right) + \left(R_h^m U_h^n - U_m^n\right) := \rho^n+\theta^n.
\end{equation*}
Then Lemma \ref{lem:Rhm} yields the following bound on $\rho^n$ as
\begin{equation}\label{eqn:rho-pod}
\frac{1}{N} \sum_{n=1}^N \|\rho^n  \|_{L^2\II}^2 \le c\sum_{j=m+1}^r  \widetilde \la_j
\qquad\text{and}\qquad
\frac{1}{N} \sum_{n=1}^N \| \rho^n  \|_{L^2\II}^2 \le c h^{-2}\sum_{j=m+1}^r  \widehat \la_j,
\end{equation}
for the $H_0^1\II$- and $L^2\II$-POD basis, respectively.
Next we derive an estimate on the component $\theta^n$. Using \eqref{eqn:fully-reduced}, the definition of the
Ritz projection operator $R_h^m$, and the fact that $\fy_m \in X_h^m \subset X_h $, we have
\begin{equation*}
\begin{split}
    (\bar\partial_\tau^\alpha \theta^n , \fy_m) + (\nabla \theta^n, \nabla \fy_m)
    &= (\bPtau^\alpha R_h^m U_h^n  , \fy_m) + (\nabla R_h^m U_h^n, \nabla \fy_m) - (\bPtau^\alpha U_m^n  , \fy_m) - (\nabla U_m^n, \nabla \fy_m)\\
    &=(\bPtau^\alpha R_h^m U_h^n  , \fy_m) + (\nabla U_h^n, \nabla \fy_m) - (f(t_n) , { \fy_m})\\
    &= (\bPtau^\alpha (R_h^m U_h^n - U_h^n) , \fy_m) = -(\bPtau^\alpha \rho^n , \fy_m)
\end{split}
\end{equation*}
and {$ \theta^0= 0 $}. The stability result in Theorem \ref{thm:stab-uni} yields
\begin{equation*}
 \|  \theta^n   \|_{L^2\II}  \le c\tau^\al \sum_{k=0}^{n-1} (n-k)^{\al-1}\|  \bPtau^\alpha \rho^{k+1}  \|_{L^2\II} .
\end{equation*}
Appealing to Young's inequality for the Laplace type discrete convolution \cite[Theorem 20.18]{HewittRoss:1979}, i.e.,
\begin{equation}\label{eqn:Young}
  \sum_{n=0}^N\left(\sum_{k=0}^{n}a_{n-k}b_{k}\right)^2\leq \left(\sum_{n=0}^Na_n\right)^2\sum_{n=0}^Nb_n^2,
\end{equation}
we deduce
\begin{equation*}
\begin{split}
  \sum_{n=1}^N  \bigg(\sum_{k=0}^{n-1} (n-k)^{\al-1}\|  \bPtau^\alpha \rho^{k+1} \|_{L^2\II} \bigg)^2
  &\le \bigg(\sum_{n=1}^N n^{\al-1}\bigg)^2 \sum_{n=1}^N \|  \bPtau^\alpha \rho^n \|_{L^2\II}^2
   \le cN^{2\al} \sum_{n=1}^N \|  \bPtau^\alpha \rho^{n} \|_{L^2\II}^2.
\end{split}
\end{equation*}
Then by Lemma \ref{lem:Rhm}, we have
\begin{equation*}
\begin{split}
  \frac{1}{N}\sum_{n=1}^N \|  \theta^n   \|_{L^2\II}^2
  &\le \frac{c \tau^{2\al}N^{2\al}}{N} \sum_{n=1}^N \|  \bPtau^\alpha \rho^n \|_{L^2\II}^2
  = \frac{c T^{2\al} }{N}\sum_{n=1}^N \|  \bPtau^\alpha \rho^n \|_{L^2\II}^2 \le c_T\sum_{j=m+1}^r \widetilde \la_j.
\end{split}
\end{equation*}
Likewise, for the $L^2\II$-POD basis, we deduce
\begin{equation*}
  \frac{1}{N}\sum_{n=1}^N \|  \theta^n   \|_{L^2\II}^2
  \le \frac{c T^{2\al} }{N}\sum_{n=1}^N \|  \bPtau^\alpha \rho^n \|_{L^2\II}^2 \le c_T h^{-2}\sum_{j=m+1}^r  \widehat \la_j.
\end{equation*}
This completes the proof of the theorem.
\end{proof}

The error estimate in Theorem \ref{thm:err-pod} covers only smooth initial data $v\in D(A)$. In the case of nonsmooth initial
data $v\in D(A^\sigma)$, $0<\sigma< 1$, one can derive an analogous error estimate; see the following remark. We note
that the regularity of problem data (or solution) does not enter the error estimate due to
the POD approximation directly. Hence, in principle, the approach is capable of handling nonsmooth problem data, if the solution
singularity is built-in in the ensemble of snapshots and thus captured by the POD basis directly.
\begin{remark}
We comment on nonsmooth problem data. Consider the  $H_0^1\II$ POD for
$f\in W^{2,\infty}(0,T;L^2\II)$ and nonsmooth initial data $v\in D(A^\sigma)$, $0<\sigma<1$. Then in view of Theorem
\ref{thm:error-fully}, we have
\begin{equation*}
  \frac{1}{N} \sum_{n=1}^N \| u(t_n)-U_h^n \|_{L^2\II}^2 \le
  c\big(\tau^{2\sigma\al} + h^4\ell_h^4  \frac{1}{N} \sum_{n=1}^N t_n^{-2\al(1-\sigma)}\big).
\end{equation*}
Meanwhile, the summation can be bounded as
\begin{equation*}
  \begin{aligned}
  \frac{1}{N} \sum_{n=1}^N t_n^{-2\al(1-\sigma)} & = \frac{\tau^{-2\alpha(1-\sigma)}}{N}\sum_{n=1}^Nn^{-2\alpha(1-\sigma)}\leq
  \frac{\tau^{-2\alpha(1-\sigma)}}{N}\int_1^N s^{-2\alpha(1-\sigma)}ds \leq c_T\ell_{\alpha,\sigma,\tau},
  \end{aligned}
\end{equation*}
where the constant $\ell_{\alpha,\sigma,\tau}$ is given by
\begin{equation*}
  \begin{aligned}
  \ell_{\alpha,\sigma,\tau}= \left\{\begin{aligned}
    &\tau^{1-2\alpha(1-\sigma)}, & \quad \alpha(1-\sigma)> 1/2, \\
    &\log\tfrac{T}{\tau}, & \quad \alpha(1-\sigma)=1/2,\\
    &1, &\quad \alpha(1-\sigma)<1/2.
  \end{aligned}\right.
  \end{aligned}
\end{equation*}
Consequently, by repeating the arguments in Theorem \ref{thm:err-pod}, we obtain the following error estimate
for the POD approximation $\{U_m^n\}$ (with the $H_0^1\II$ POD basis)
\begin{equation*}
\frac{1}{N} \sum_{n=1}^N \| u(t_n)-U_m^n \|_{L^2\II}^2
\le c_T\bigg(\tau^{2\sigma\al} + h^4\ell_h^4\ell_{\alpha,\sigma,\tau}  + \sum_{j=m+1}^r\widetilde \la_j \bigg),
\end{equation*}
and a similar error estimate holds for the $L^2\II$ POD basis. Interestingly, for the case $\alpha(1-\sigma)<1/2$,
the error estimate in the space remains uniform with respect to the time step size $\tau$.
\end{remark}

Last we briefly comment on the case when the FDQs are not
included in the snapshots.
\begin{remark}\label{rmk:DQ}
In our construction of the POD basis, we have included the FDQs in the snapshots.
When the FDQs $\bar\partial_\tau^\alpha U_h^n$, $n=1,2,...,N$, are not contained in the snapshot set,
the error formula \eqref{eqn:eh} for $H_0^1\II$ POD basis becomes
\begin{equation*}
  \frac1{N+1}  \sum_{n=0}^N \| U_h^n- \sum_{j=1}^m (\nabla U_h^n,\nabla \wpsi_j) \wpsi_j \|_{H_0^1\II}^2
  = \sum_{j=m+1}^r \widetilde \la_j.
\end{equation*}
Further for the FDQs we have
\begin{equation*}
\begin{split}
 \frac1N\sum_{n=1}^N \| \bPtau^\al  U^n_h- \sum_{j=1}^m (\nabla \bPtau^\al U_h^n, \nabla \wpsi_j) \wpsi_j\|_{H_0^1\II}^2
 =&   \frac1N\sum_{n=1}^N \| \bPtau^\al  \big(U^n_h- \sum_{j=1}^m (\nabla U_h^n, \nabla \wpsi_j) \wpsi_j\big) \|_{H_0^1\II}^2
\end{split}
\end{equation*}
Let $\overline U_h^n = U^n_h- \sum_{j=1}^m (\nabla U_h^n, \nabla \wpsi_j) \wpsi_j$. By
the monotonicity of the weights $\{b_j\}$, we have
\begin{equation*}
\begin{split}
  \left\| \bPtau^\al  \overline{U}_h^n\right\|_{H_0^1\II}^2 &\leq
  c_\alpha\tau^{-2\alpha}\bigg(b_0\|\overline{U}_h^n\|_{H_0^1\II} + b_{n-1}\|\overline{U}_h^0\|_{H_0^1\II} + \sum_{j=1}^{n-1}(b_{j-1}-b_j)\|\overline{U}_h^{n-j}\|_{H_0^1\II}\bigg)^2\\
  &\le c_\alpha\tau^{-2\alpha}b_n^2\|\overline{U}_h^0\|_{H_0^1\II}^2+ c_\alpha\tau^{-2\alpha}\bigg( \sum_{j=0}^{n}g_j\|\overline{U}_h^{n-j}\|_{H_0^1\II} \bigg)^2,
\end{split}
\end{equation*}
with $g_j=b_{j-1}-b_{j}$ and $b_{-1}=2$. Then by Young's inequality for discrete convolution, cf. \eqref{eqn:Young}, we arrive at
\begin{equation*}
\frac1N\sum_{n=1}^N \bigg( \sum_{j=0}^ng_j\|\overline U_h^{n-j}\|_{H_0^1\II} \bigg)^2
\le \frac1N \bigg(\sum_{n=0}^N g_j\bigg)^2 \sum_{n=0}^N\|\overline{ U}_h^{n-j}\|_{H_0^1\II}^2
\le \frac{c}{N+1}\sum_{n=0}^N\|\overline{U}_h^n\|_{H_0^1\II}^2.
\end{equation*}
Meanwhile, by the Cauchy-Schwarz inequality, we have
\begin{equation*}
  \begin{aligned}
   \sum_{n=1}^N b_n^2 & = (1-\al)^{-2}\sum_{n=1}^N \bigg(\int_n^{n+1} s^{-\al} \,ds\bigg)^2
   &\le c \int_1^{N+1} s^{-2\al} \,ds \le \left\{\begin{aligned} cN^{1-2\al}, &\quad \mbox{if } \alpha<1/2, \\
   c\log N, &\quad \mbox{if } \alpha = 1/2, \\
   c,        &\quad \mbox{if } \alpha >1/2.
   \end{aligned}\right.
  \end{aligned}
\end{equation*}
Consequently, there holds
\begin{equation*}
  \frac1N\sum_{n=1}^N \| \bPtau^\al  U^n_h- \sum_{j=1}^m (\nabla \bPtau^\al U_h^n, \nabla \wpsi_j) \wpsi_j\|_{H_0^1\II}^2
  \le  c_T\big(\ell_{\alpha,\tau}\|\overline{U}_h^0\|_{H_0^1\II}^2 + \tau^{-2\alpha} \sum_{j=m+1}^r  \widetilde \la_j \big),
\end{equation*}
where the constant $\ell_{\alpha,\tau}$ is given by
\begin{equation*}
  \ell_{\alpha,\tau} = \left\{\begin{aligned}
    1 &\quad \mbox{ if } \alpha < 1/2, \\
    \log \tfrac{T}{\tau} &\quad \mbox{ if } \alpha = 1/2,\\
    \tau^{-2\alpha + 1} &\quad \mbox{ if } \alpha>1/2.
  \end{aligned}\right.
\end{equation*}
For $\alpha\leq 1/2$, the term involving the initial data $U_h^0$ is of higher order in comparison with the last term.
Hence the error for $H_0^1\II$ Galerkin POD \eqref{eqn:fully-reduced} (without FDQs in the snapshots)
can be bounded by
\begin{equation*}
\frac{1}{N} \sum_{n=1}^N \| u(t_n)-U_m^n \|_{L^2\II}^2
\le c_T\bigg(\tau^{2\al} + h^4\ell_h^4  + \ell_{\alpha,\tau}\| U^0_h- \sum_{j=1}^m (\nabla U_h^0, \nabla \wpsi_j) \wpsi_j\|_{H_0^1\II}^2  +  \tau^{-2\alpha} \sum_{j=m+1}^r \widetilde \la_j \bigg).
\end{equation*}
In comparison with the error estimate \eqref{eqn:err-pod-1} with FDQs
from Theorem \ref{thm:err-pod}, this estimate contains an extra factor $\tau^{-2\alpha}$ and
an approximation error of the initial data $v_h$ (within the POD basis $X_h^m$). For the fractional order
$\alpha\to1$, the factor recovers that for the classical diffusion equation \cite{KunischVolkein:2001}.
\end{remark}

\section{Numerical results}\label{sec:numeric}
Now we present numerical results to verify the convergence theory in Section \ref{sec:conv} and
the efficiency of the proposed Galerkin-L1-POD scheme.

\subsection{Numerical results for one-dimensional examples}
First we present numerical results for one-dimensional examples to verify the convergence
analysis in Section \ref{sec:conv}. We consider the subdiffusion model in the
following two cases:
\begin{itemize}
  \item[(a)] $\Omega=(0,1)$,  $v=x(1-x) \in D(A)$, and $f(x,t)=e^{t\cos(2\pi x)}\in W^{2,\infty}(0,T;L^2\II)$;
  \item[(b)] $\Omega=(0,1)$,  $v=\chi_{(0,1/2)}(x)\in  D(A^{1/4-\ep})$ for $\epsilon\in(0,1/4)$, and $f(x,t)=e^{t\cos(2\pi x)}\in W^{2,\infty}(0,T;L^2\II)$.
\end{itemize}
In the computations, we divide the unit interval $\Omega$ into $M$ equally spaced subintervals
with a mesh size $h=1/M$. Likewise, we fix the time step size $\tau$ at $\tau=T/N$.

First we examine the temporal convergence by setting $T=0.1$ (the spatial convergence was already examined
in \cite{JinLazarovZhou:2013,JinLazarovPasciakZhou:2013a}). We take a small mesh size $h=10^{-3}$,
so that the spatial discretization error is negligible. The exact solution can be expressed
in terms of the Mittag-Leffler function $E_{\al,\beta}(z)$, cf. \eqref{op:E}, which
can be evaluated efficiently by an algorithm
developed in \cite{Seybold:2008}.  The numerical results by the fully discrete scheme \eqref{eqn:fully}
are given in Table \ref{tab:fde-uniform}. In the table, \texttt{rate} refers to
the empirical rate when the time step size $\tau$ halves, and the numbers in the bracket denote the
theoretical predictions from Theorem \ref{thm:error-fully}. For cases (a)
and (b), the empirical rate is $O(\tau^\alpha)$ and $O(\tau^{\alpha/4})$, respectively,
which agree well with the theoretical ones. The convergence rate of the L1 scheme improves with
the smoothness of the initial data $v$ (while keeping the
smooth right hand side $f$ fixed) and the increase of the fractional order $\alpha$, since
 the solution regularity improves accordingly.

\begin{table}[htb!]
\caption{The maximum error $ e_{\max} = \max_{1\le n\le N} \|  U_h^n - u(t_n)\|_{L^2\II}$ for initial data (a)
and (b) with $T=0.1$, $h=10^{-3}$, $\tau=T/N$.}\label{tab:fde-uniform}
\begin{center}
\vspace{-.3cm}{\setlength{\tabcolsep}{7pt}
     \begin{tabular}{|c|c|cccccc|c|}
     \hline
      $\alpha$ &  $N$    &$1000$ &$2000$ & $4000$ & $8000$ &$16000$&$32000$
      &rate \\
     \hline
     $0.35$ & (a)  &2.67e-3 &2.27e-3 &1.90e-3 &1.58e-3 &1.29e-3 &1.05e-3 &
     $\approx$ 0.29 (0.35)\\
             &  (b)   &2.48e-2 &2.41e-2 &2.29e-2 &2.15e-2 &1.99e-2 &1.82e-2  &
             $\approx$ 0.10 (0.09)\\
      \hline
     $0.5$ &  (a)   &9.26e-4 &6.73e-4 &4.86e-4 &3.50e-4 &2.51e-4 &1.80e-4 &
     $\approx$ 0.48 (0.50)\\
             &  (b)   &2.03e-2 &1.81e-2 &1.64e-2 &1.50e-2 &1.37e-2 &1.26e-2  &
             $\approx$ 0.13 (0.13)\\
      \hline
     $0.75$ &  (a)    &1.82e-4 &1.09e-4 &6.43e-5 &3.77e-5 &2.17e-5 &1.25e-5   &$\approx$ 0.76 (0.75)\\
             &  (b)    &2.52e-2 &2.20e-2 &1.91e-2 &1.64e-2 &1.39e-2  & 1.15e-2  &$\approx$ 0.21 (0.19)\\
     \hline
     \end{tabular}}
\end{center}
\end{table}

Next we illustrate the proposed Galerkin-L1-POD scheme, and the numerical results are
given in Table \ref{tab:semi_c} for the choice $T=1$ and $N=200$. Here the average error
$e$ and the POD approximation error $e^m$ are defined by
\begin{equation*}
  e = \frac{1}{N}\sum_{n=1}^N \| U_h^n -u(t_n) \|_{L^2\II}^2\quad \mbox{and}\quad   e^m =\frac{1}{N}\sum_{n=1}^N \| U_h^n -U_m^n \|_{L^2\II}^2,
\end{equation*}
respectively. Like before, we use the notation $~~\widetilde{}~~$ and $~~\widehat{}~~$ over $e^m$ to denote $H_0^1\II$-
and $L^2\II$-POD basis, respectively, and the subscript $w$ to indicate that the snapshots do not contain FDQs.
For example, $\widetilde{e}^m$ and $\widetilde{e}_w^m$ denote the error between the full Galerkin
solution $U_h^n$ and the solution of the Galerkin POD formulation with $m$ $H_0^1\II$ POD basis functions, with and
without FDQs, respectively. For both cases (a) and (b), with three or four POD basis functions, the POD approximation
error falls below the error due to temporal discretization, and the convergence is relatively independent of the
fractional order $\alpha$. The fast convergence of the Galerkin POD scheme is also expected from the exponential
decay of the eigenvalues of the correlation matrix, cf. Fig. \ref{fig:1d}. Further, the inclusion of FDQs does not
affect much the POD approximation error, with their errors within a factor of ten, even though their presence improves the apparent theoretical
convergence rates, cf. Theorem \ref{thm:err-pod} and Remark \ref{rmk:DQ}. The effect seems to be compensated by the
smaller eigenvalues, cf. Fig. \ref{fig:1d}. These observations show the efficiency of the Galerkin POD scheme, which
has only a degree of freedom of three or four at each time level, compared with one thousand for the standard Galerkin FEM.

For case (b), the Galerkin POD scheme requires slightly more POD
basis functions in order to reach the same level of the accuracy. This is expected, since for nonsmooth
data $v$, it can only be accurately described by more Fourier modes, and all these modes persist in the dynamics
due to the ``slow'' decay of subdiffusion. Hence the solution manifold may exhibit richer structure
than case (a), and consequently, more POD basis functions are needed to accurately capture the dynamics.
However, the eigenvalues in the nonsmooth case decays also exponentially, cf. Fig. \ref{fig:1d}. Hence, the
proposed scheme also works well  with low regularity data.

The efficiency of the proposed scheme relies crucially on constructing ``good'' POD basis. To this end, we
present the first five POD basis functions for case (b) in Fig. \ref{fig:podbasis}. The $H_0^1\II$- and
$L^2\II$ POD basis take very different shapes: for the $H_0^1\II$ POD, the first basis function captures the singularity
(caused by the discontinuous initial data), whereas the higher POD modes are very smooth. In contrast, for the $L^2\II$
POD, all the first five POD basis functions contain singularities (in the middle of the interval as well as oscillations
around the end points). Namely, the $H_0^1\II$ POD seems to better aggregate the solution singularity (actually into one single
POD basis). Nonetheless, the $L^2\II$ and $H_0^1\II$ POD-basis exhibit quite similar approximation property, and thus
can provide equally good approximations of the solution manifold, cf. Table \ref{tab:semi_c}.

\begin{table}[h]
\caption{The numerical results of the Galerkin POD for cases (a) and (b) with $T=1$,
$h=10^{-3}$, $N=200$, and with $m$ POD basis functions. }\label{tab:semi_c}
\vspace{-.2cm}
\begin{center}
\begin{tabular}{|c|c|c|c|c|c|c|c|}
\hline
   $\al$& case  & $m$ & $e$ & ${\widetilde e}^m$ & ${\widetilde e}^m_w$ &$\widehat e^m$ &$\widehat e^m_w$ \\ \hline
 & (a) &3  &1.82e-7 & 9.34e-12 & 3.03e-12 &9.45e-12  &3.02e-12  \\ 
 & &4  &1.82e-7 & 4.72e-13 & 3.71e-14 &4.83e-13  &3.19e-14  \\  \cline{2-8}
0.3 & (b) &3  &3.83e-6 & 4.65e-6 & 3.59e-6 &4.36e-6  &3.60e-6  \\
& &4 &3.83e-6 & 2.73e-9 & 2.41e-9 &2.73e-9  &2.41e-9  \\  \hline
 & (a) &3  &4.46e-7& 1.01e-10 & 6.25e-12 &1.11e-10  &6.22e-12  \\  
 & &4  &4.46e-7 & 5.33e-13 & 8.87e-14 &5.41e-13  &8.28e-14  \\  \cline{2-8}
0.5 & (b) &3  &1.70e-5 & 1.81e-5 & 6.70e-6 &1.59e-5  &7.08e-6  \\
& &4 &1.70e-5 & 3.67e-8 & 6.70e-9 &3.43e-8  &6.69e-9  \\  \hline
 & (a) &3  &2.89e-7 &4.70e-10 & 1.35e-11 &4.98e-10  &1.34e-11  \\ 
 & &4  &2.89e-7 & 1.33e-12 & 1.85e-13 &1.29e-12  &1.81e-13  \\  \cline{2-8}
0.7 & (b) &4  &2.80e-5 & 2.51e-5 & 1.83e-7 &1.45e-5  &1.78e-7  \\
& &5 &2.80e-5 & 2.49e-8 & 5.00e-9 &2.42e-8  &4.99e-9  \\  \hline
\end{tabular}
\end{center}
\end{table}

\begin{figure}[hbt!]
\centering

\subfigure[case (a), $\al=0.3$]{
\includegraphics[trim = .1cm .1cm .1cm .1cm, clip=true,width=4.5cm]{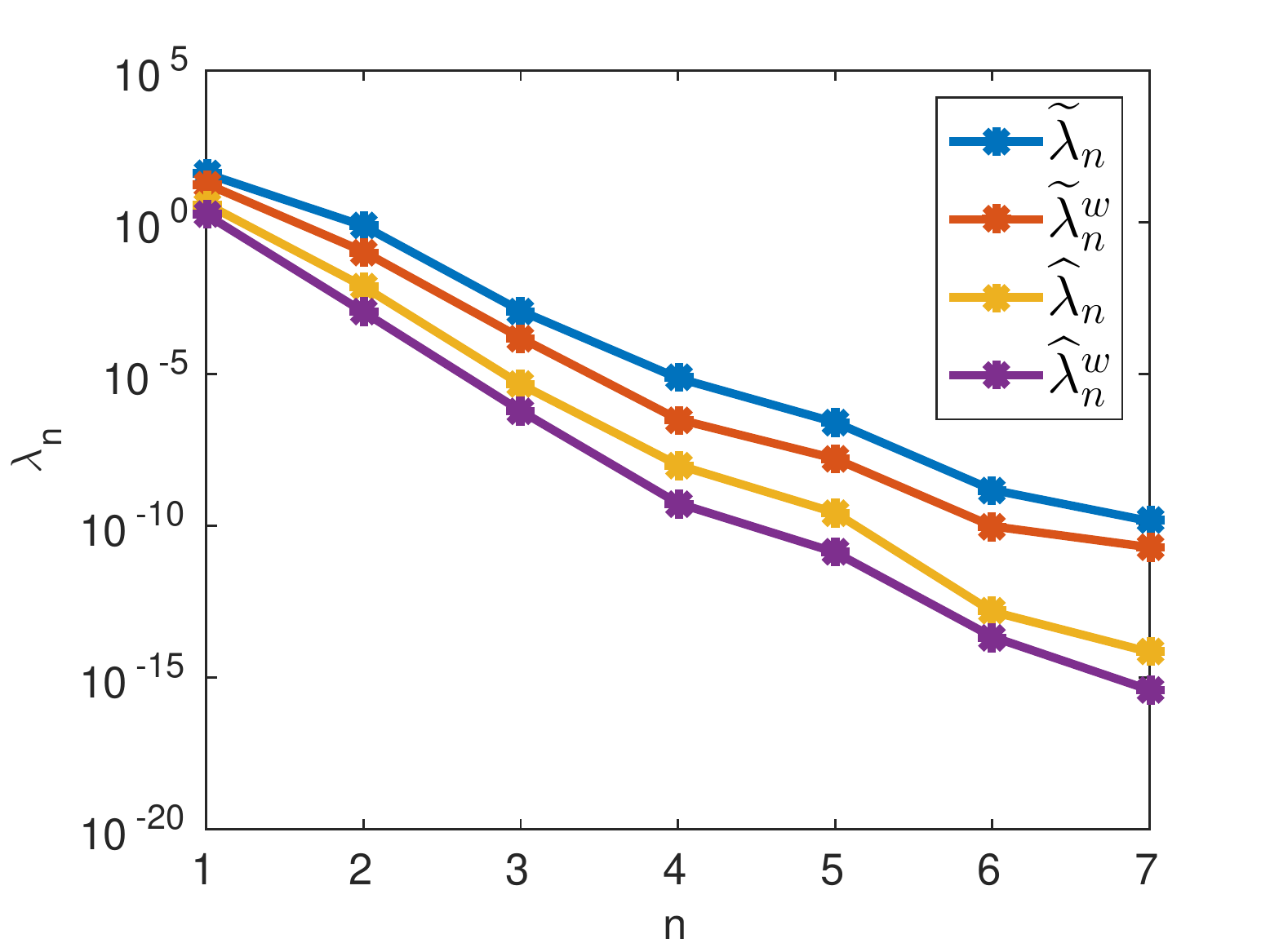}
}
\subfigure[case (a), $\al=0.5$]{
\includegraphics[trim = .1cm .1cm .1cm .1cm, clip=true,width=4.5cm]{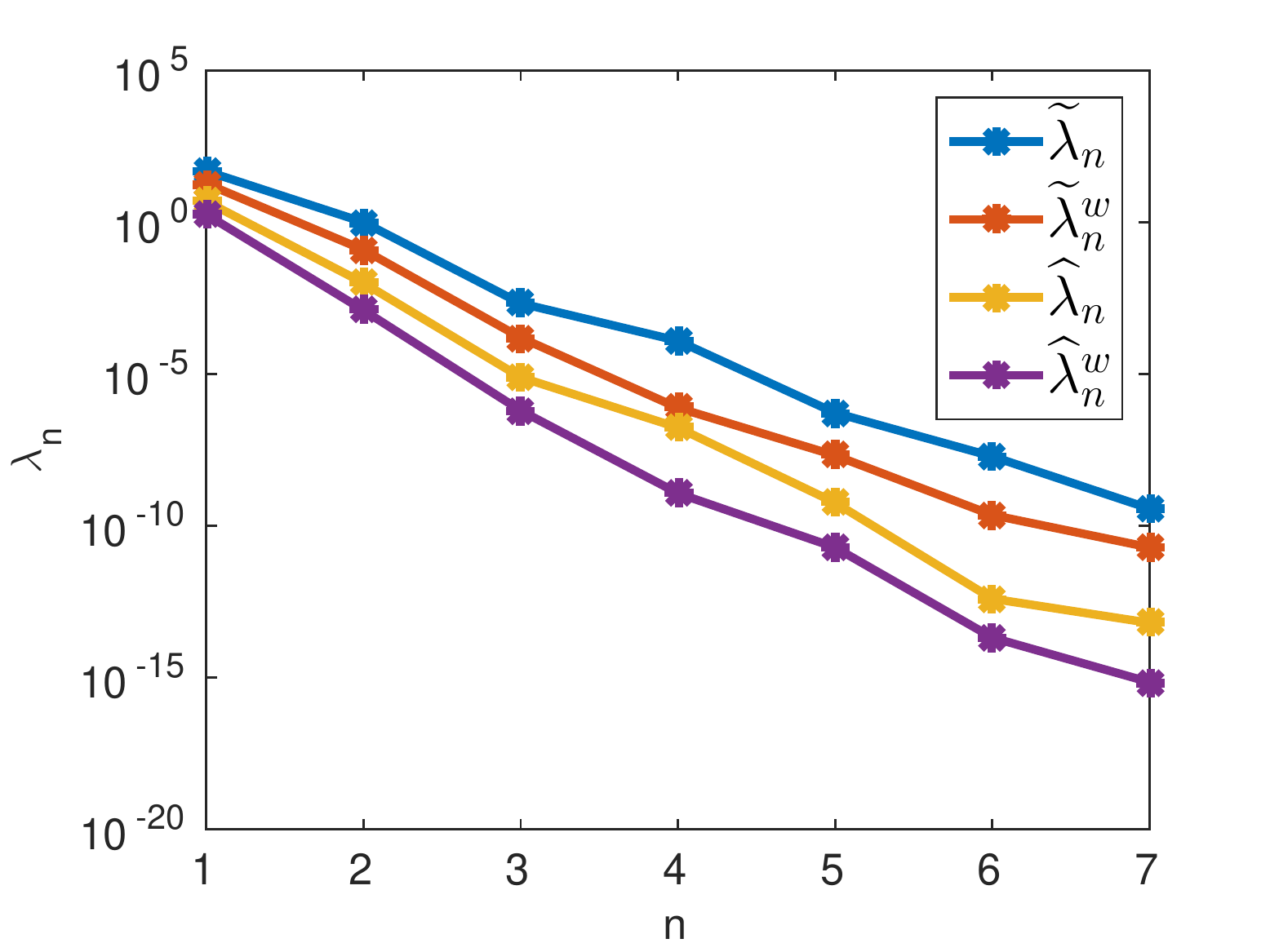}
}
\subfigure[case (a), $\al=0.7$]{
\includegraphics[trim = .1cm .1cm .1cm .1cm, clip=true,width=4.5cm]{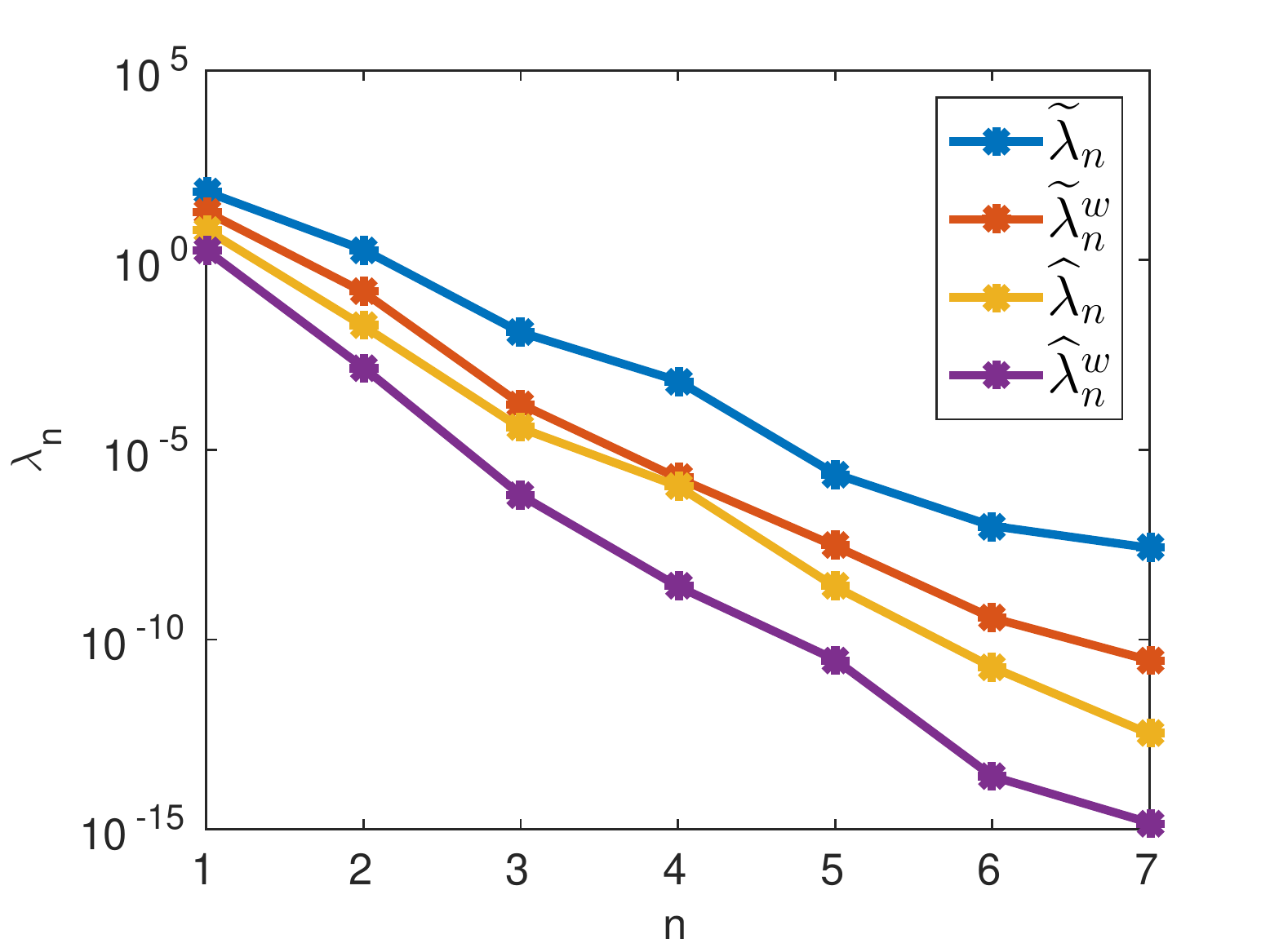}
}

\subfigure[case (b), $\al=0.3$]{
\includegraphics[trim = .1cm .1cm .1cm .1cm, clip=true,width=4.5cm]{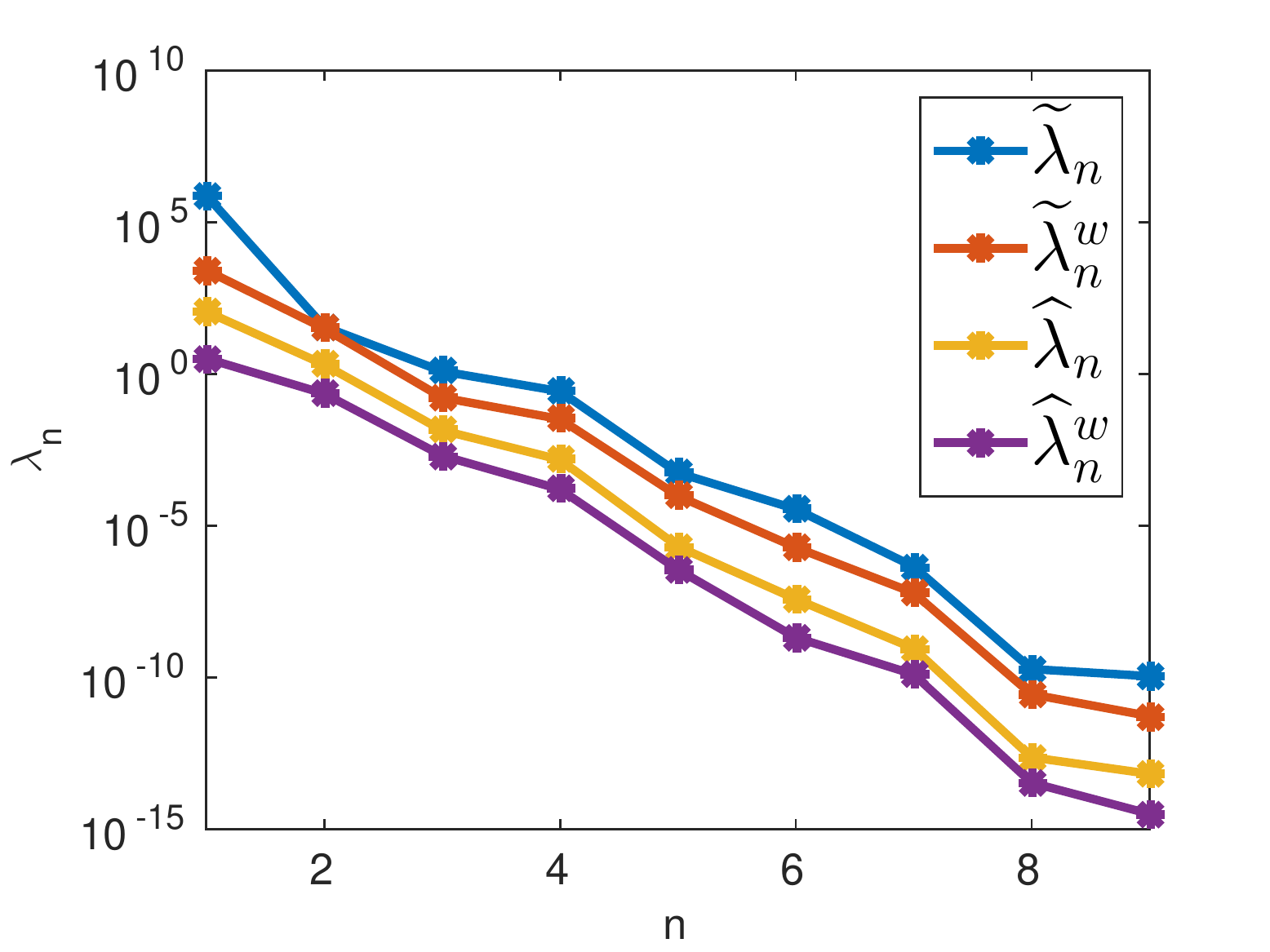}
}
\subfigure[case (b), $\al=0.5$]{
\includegraphics[trim = .1cm .1cm .1cm .1cm, clip=true,width=4.5cm]{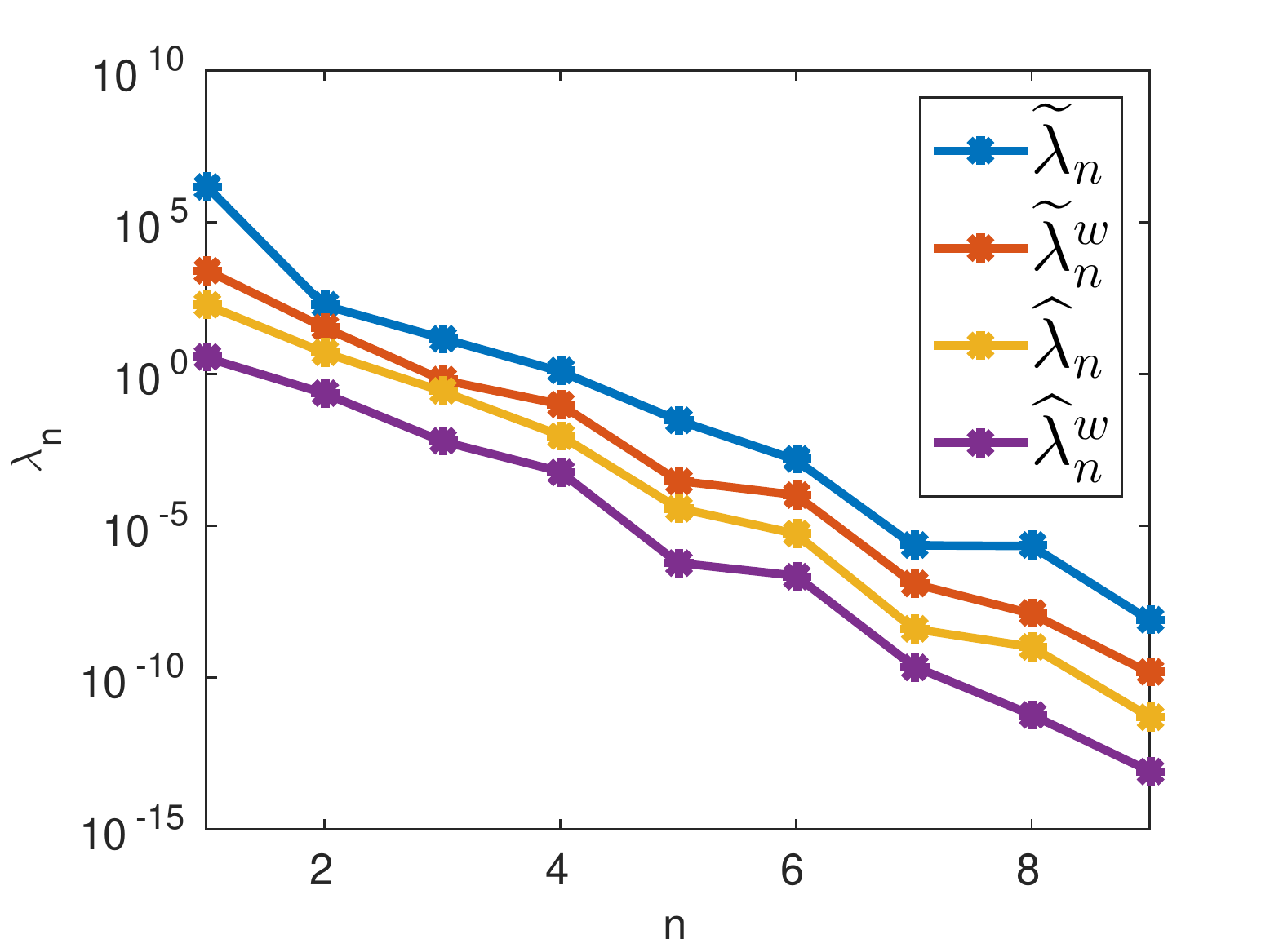}
}
\subfigure[case (b), $\al=0.7$]{
\includegraphics[trim = .1cm .1cm .1cm .1cm, clip=true,width=4.5cm]{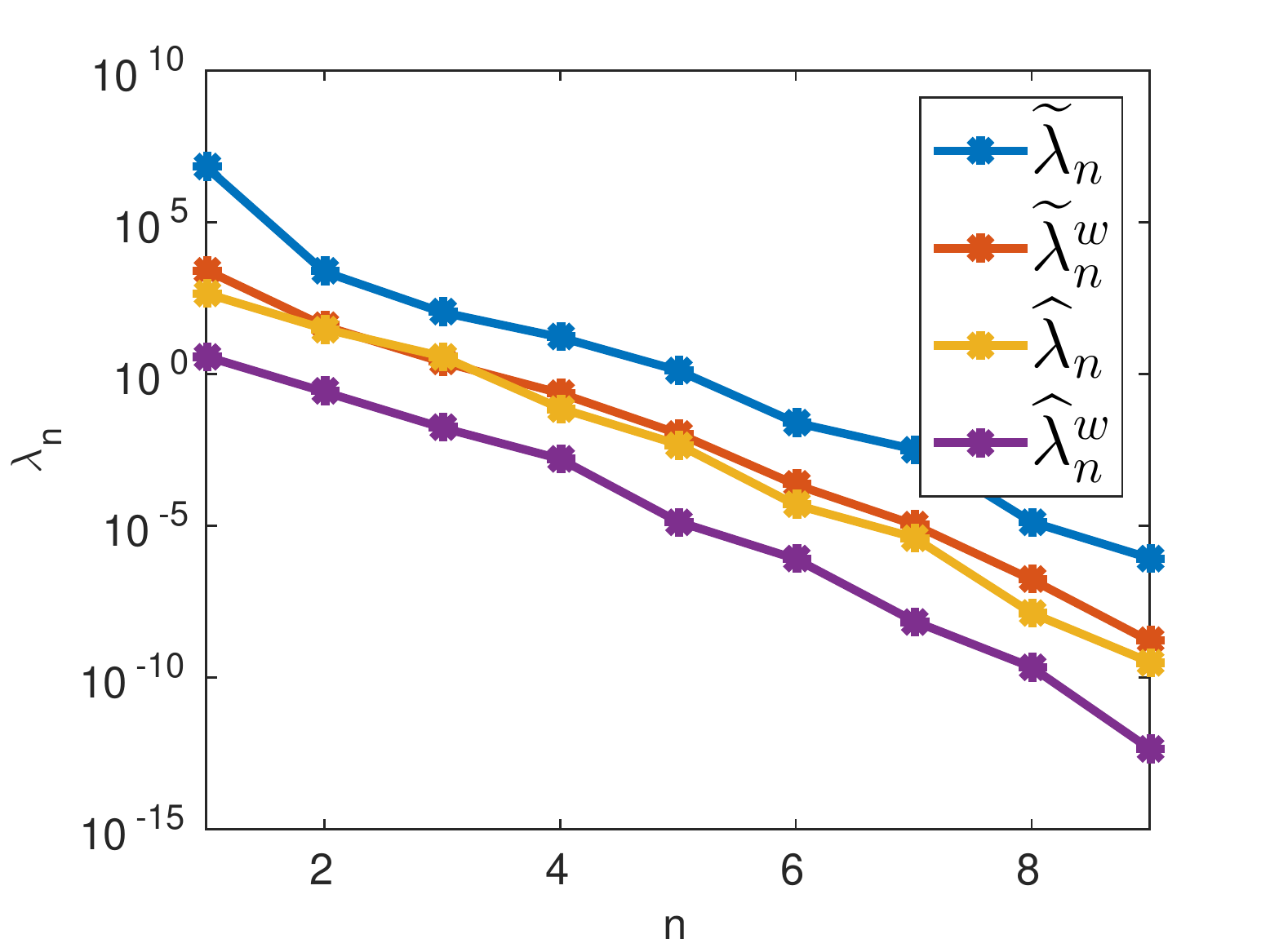}
}
 \caption{The decay of eigenvalues of the correlation matrix in the 1D problem with $\al=0.3$, $0.5$ and $0.7$.
 Here, $\widetilde\lambda_n$, $\widetilde\lambda_n^w$, $\widehat\lambda_n$, and $\widehat\lambda_n^w$
 denote eigenvalues of correlation matrix for $H_0^1\II$ POD basis with or without FDQs
 and $L^2\II$ POD basis  with or without FDQs, respectively.}
\label{fig:1d}
\end{figure}

\begin{figure}[hbt!]
\centering

\subfigure[$\al=0.3$, $H_0^1\II$ basis]{
\includegraphics[trim = .1cm .1cm .1cm .1cm, clip=true,width=4.5cm]{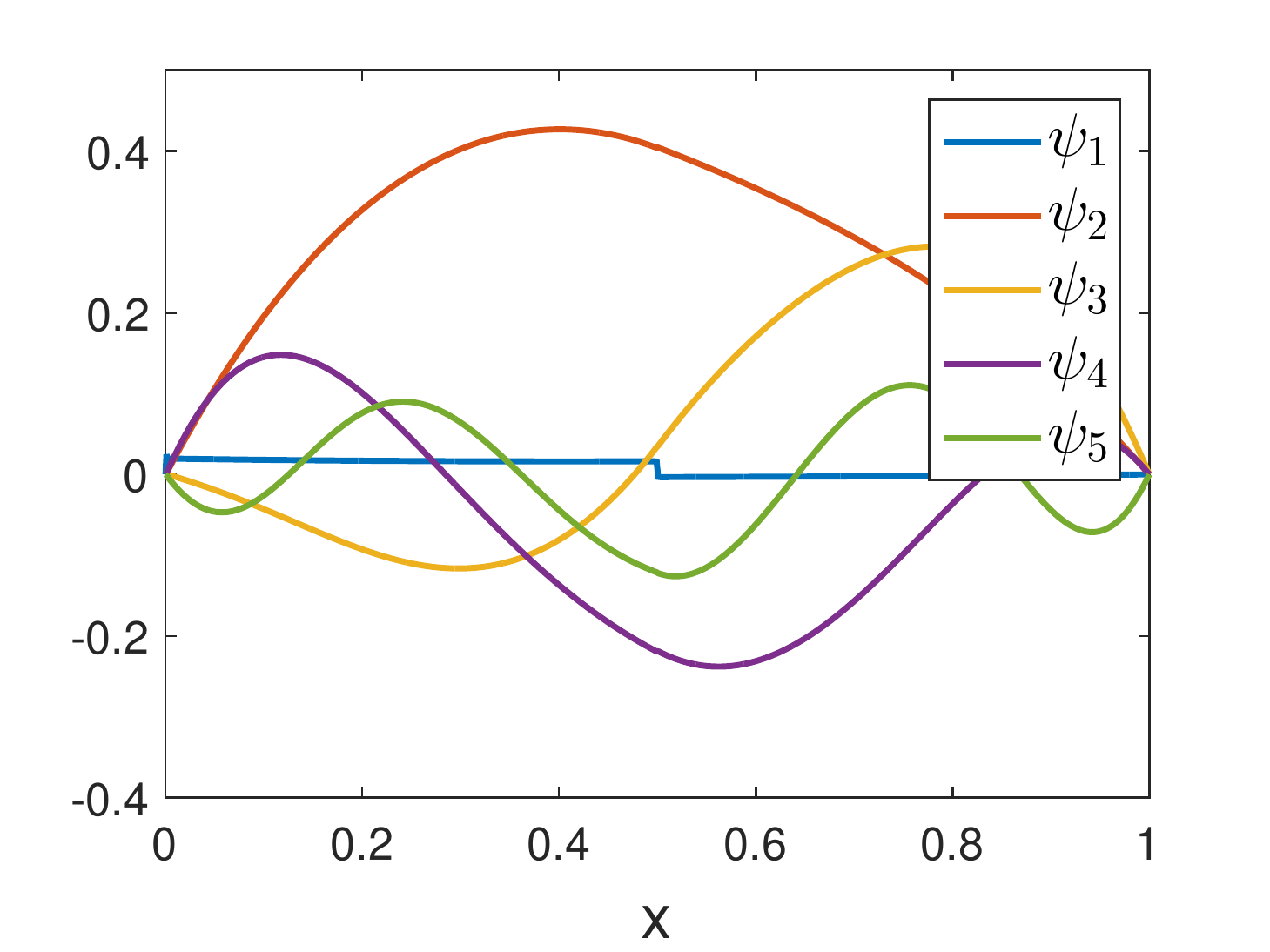}
}
\subfigure[$\al=0.5$, $H_0^1\II$ basis]{
\includegraphics[trim = .1cm .1cm .1cm .1cm, clip=true,width=4.5cm]{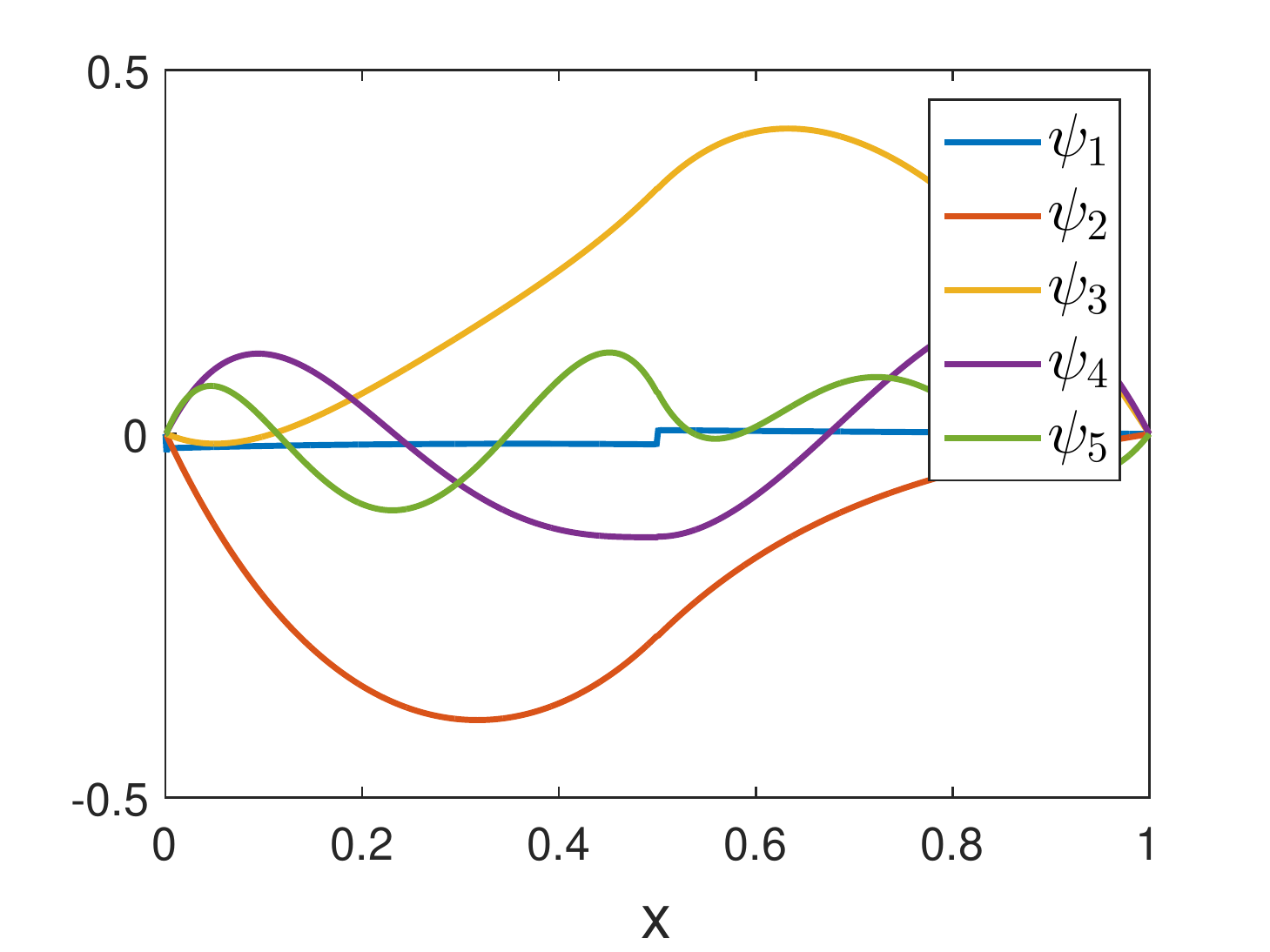}
}
\subfigure[$\al=0.7$, $H_0^1\II$ basis]{
\includegraphics[trim = .1cm .1cm .1cm .1cm, clip=true,width=4.5cm]{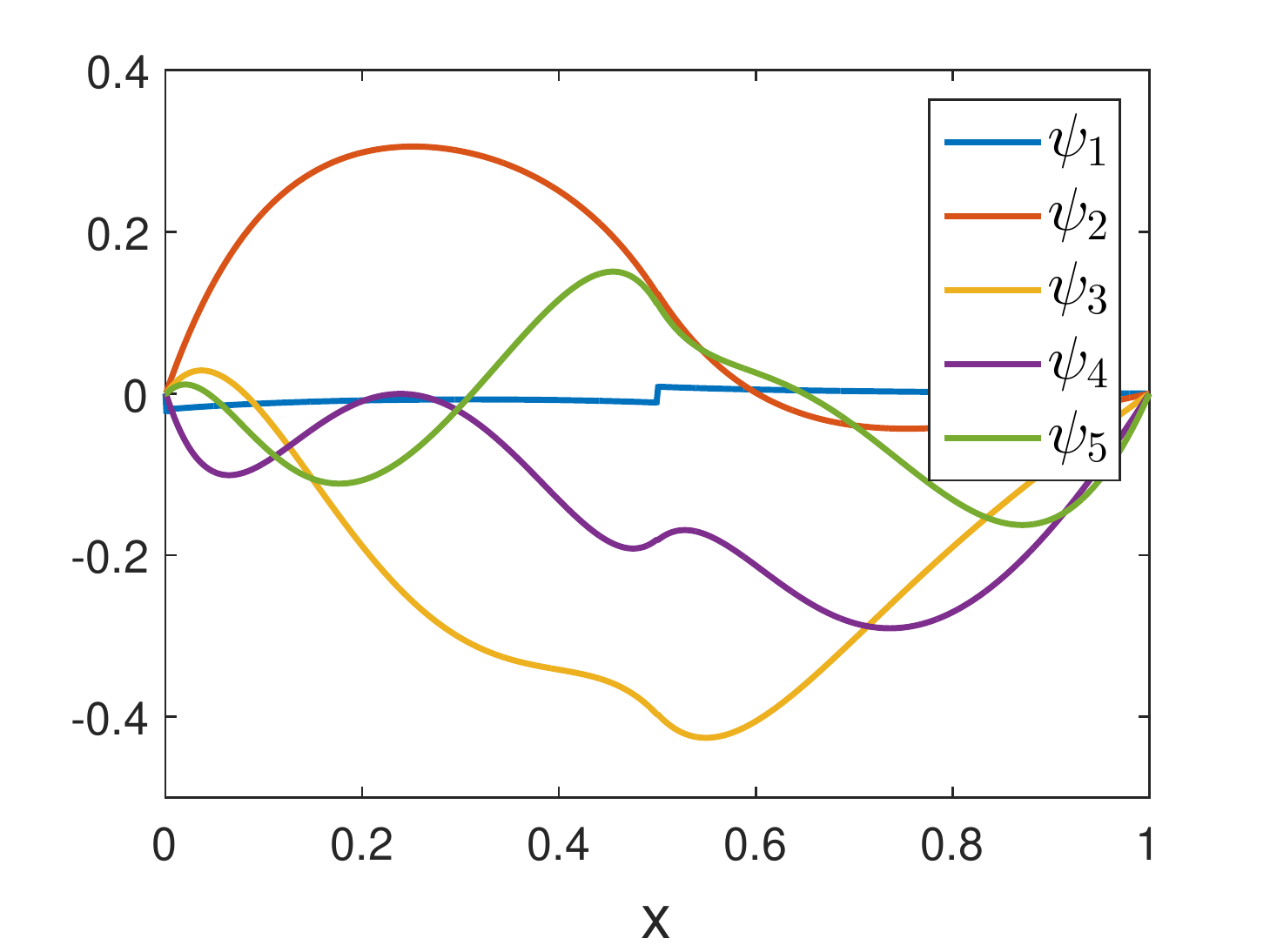}
}

\subfigure[$\al=0.3$, $L^2\II$ basis]{
\includegraphics[trim = .1cm .1cm .1cm .1cm, clip=true,width=4.5cm]{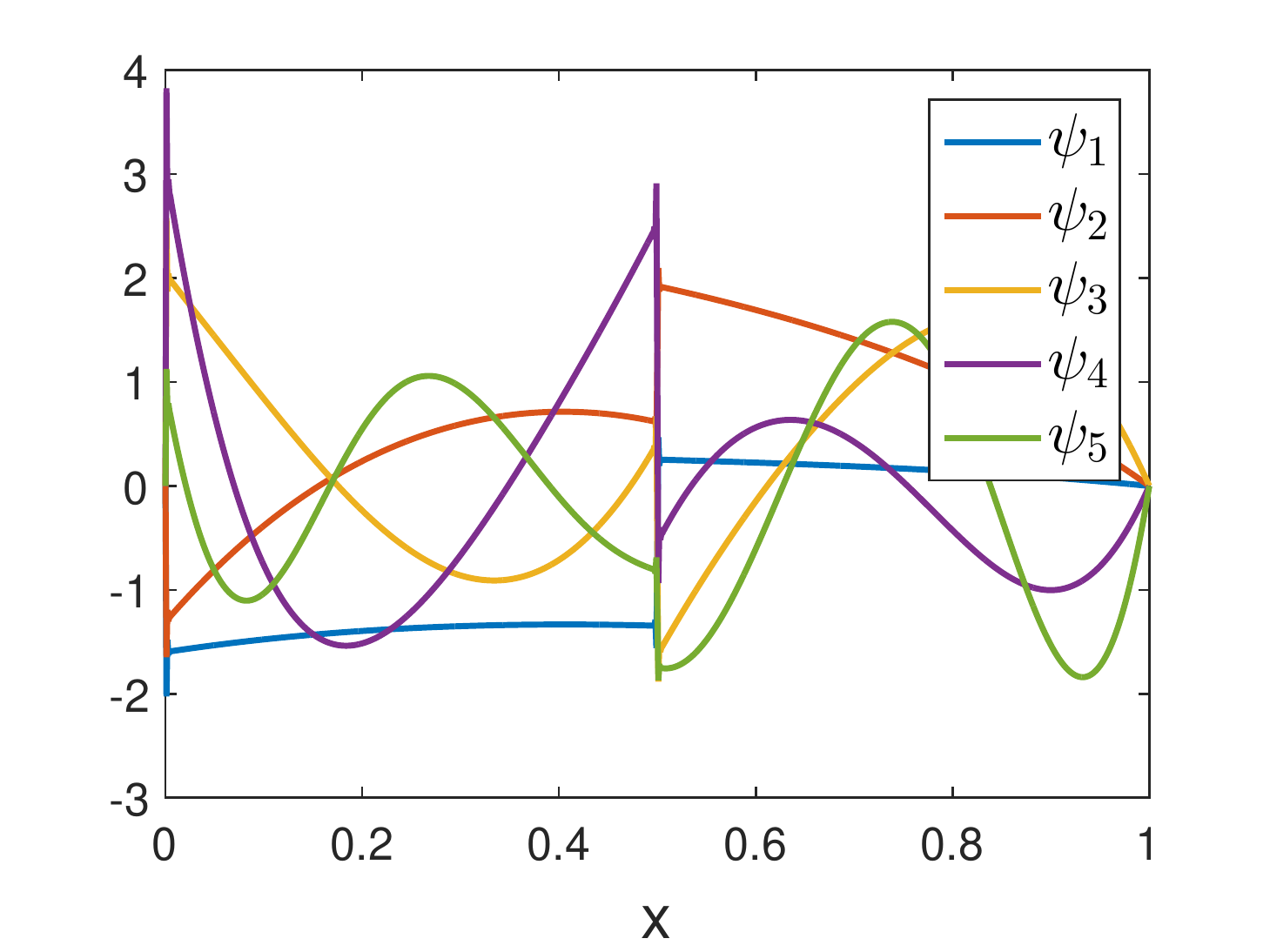}
}
\subfigure[$\al=0.5$, $L^2\II$ basis]{
\includegraphics[trim = .1cm .1cm .1cm .1cm, clip=true,width=4.5cm]{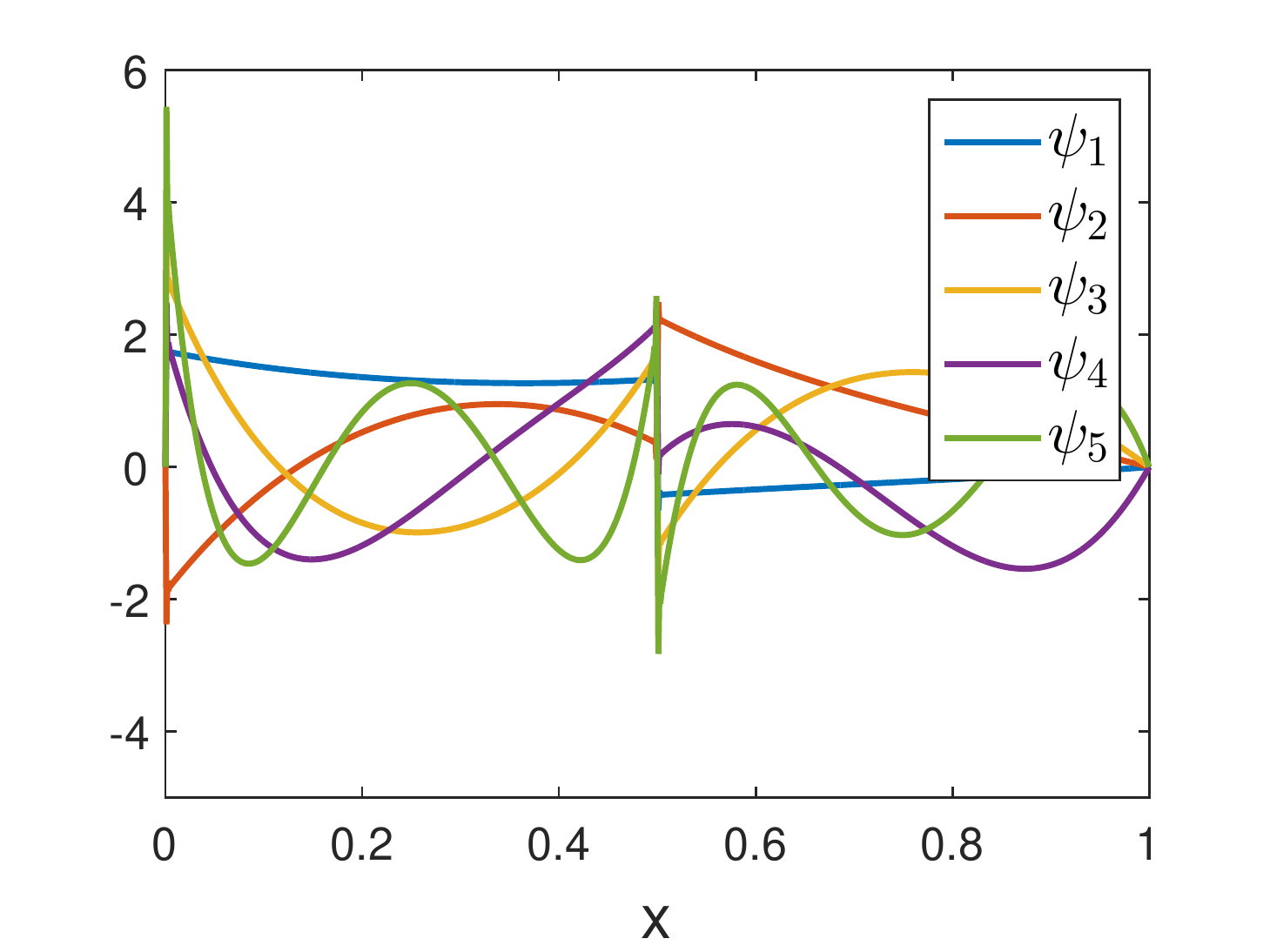}
}
\subfigure[$\al=0.7$, $L^2\II$ basis]{
\includegraphics[trim = .1cm .1cm .1cm .1cm, clip=true,width=4.5cm]{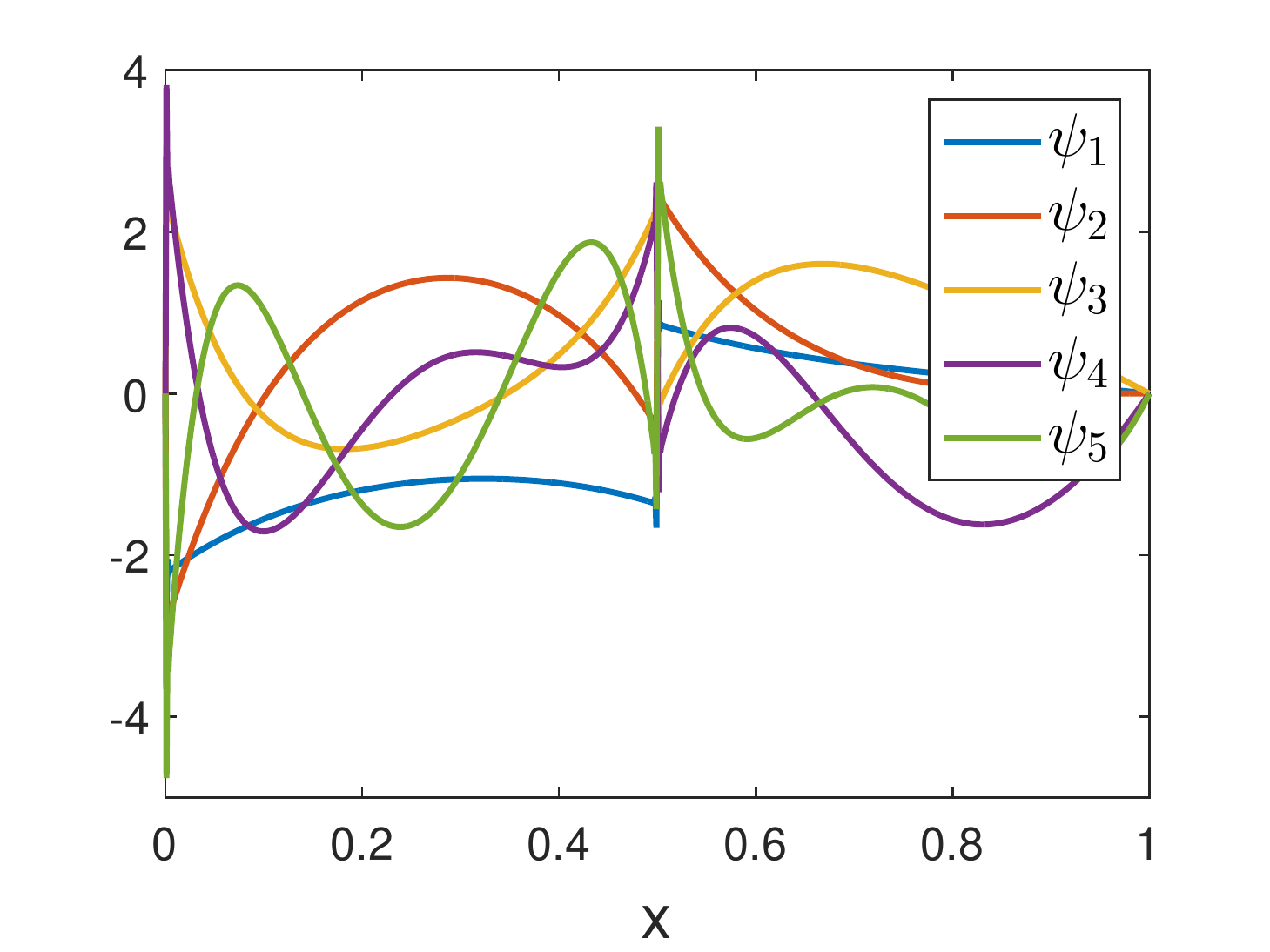}
}
 \caption{The first five POD basis functions, in the $H_0^1\II$ and $L^2\II$ norms for case (b), with FDQs included in the basis construction.}
\label{fig:podbasis}
\end{figure}

\subsection{Numerical results for one two-dimensional example}
Now we present numerical results for the following two-dimensional example:
\begin{itemize}
  \item[(c)] $\Omega=(-1,1)^2\setminus ([0,1]\times[-1,0])$, $v(x_1,x_2)=x_1(1+x_1)(1-x_1)\sin(2\pi x_2)$,
   $f(x_1,x_2,t)=e^{t\cos(2\pi x_1)\sin(\pi x_2)}\in W^{2,\infty}(0,T;L^2\II)$, and $T=1$.
\end{itemize}
In the computations, we divide the L-shaped domain $\Omega$ into a triangulation with a degree of freedom $10^{4}$,
and fix the time step size $\tau$ at $\tau=T/200$. A reentrant corner with an angle $\omega\in(\pi,2\pi)$
induces a singularity associated with the corresponding stationary Poisson's problem \cite{Grisvard:1985}.
In example (c), the angle $\omega=3\pi/2$, and the reentrant corner gives rise to a singularity near the origin
with a leading term of the form $ r^{2/3}\sin(2\theta/3)$ in polar coordinates.
Hence, we refine the mesh adaptively using the bisection rule \cite[Section 4.1]{Nochetto:2009}.
We compute the reference solution on a more refined mesh with $ 2\times10^4$ and $\tau=1/1000$.

The numerical results are shown in Table \ref{tab:POD-2D_lshape}. The POD scheme exhibits a fast
convergence, and the error decreases steadily with the increase of the number $m$ of POD basis
functions. In particular, five or six POD basis functions suffice to resolve the solution
manifold to an accuracy $O(10^{-9})$, which clearly shows the efficiency of the Galerkin POD scheme, when
compared with the standard Galerkin FEM. The fast convergence follows also from the exponential
decay of the eigenvalues of the correlation matrix, cf. Fig. \ref{fig:POD-2D_lshape}.
 The decay rate of the spectrum is almost identical for the $L^2\II$ and
$H^1\II$ POD basis, and independent of the presence of the FDQs. Hence,
the presence of geometrical singularities in the domain does not influence the efficiency of the
Galerkin-L1-POD scheme. Interestingly, we observe that with the increase of the fractional order
$\alpha$, the error increases slightly, which awaits further theoretical justification.

\begin{table}[hbt!]
\caption{The numerical results of the Galerkin POD for case (c), with $T=1$, $N=200$
and with $m$ POD basis functions.}\label{tab:POD-2D_lshape}
\vspace{-.2cm}
\begin{center}
\begin{tabular}{|c|c|c|c|c|c|c|}
\hline
   $\al$  & $m$ & $e$ & ${\widetilde e}^m$ & ${\widetilde e}^m_w$ &$\widehat e^m$ &$\widehat e^m_w$ \\ \hline
 0.3 &5  &7.67e-7 & 5.36e-10 & 3.33e-10 &5.17e-10  &3.32e-10  \\ 
 & 6  &7.67e-7 & 6.40e-12 & 5.49e-12 &6.39e-12  &5.48e-12 \\  \hline
 0.5  &5  &4.75e-6& 2.08e-8 & 8.23e-9 &1.96e-8  &8.18e-9  \\  
 & 6  &4.75e-6 & 1.62e-10 & 4.82e-11 &1.44e-10  &4.79e-11  \\  \hline
 0.7  &6  &1.01e-5 &2.05e-8 & 1.36e-9 &1.38e-8  &1.27e-9  \\ 
 & 7  &1.01e-5 & 9.11e-10 & 1.17e-10 &6.09e-10  &1.11e-10  \\  \hline
\end{tabular}
\end{center}
\end{table}

\begin{figure}[hbt!]
\centering

\subfigure[$\al=0.3$]{
\includegraphics[trim = .1cm .1cm .1cm .1cm, clip=true,width=4.5cm]{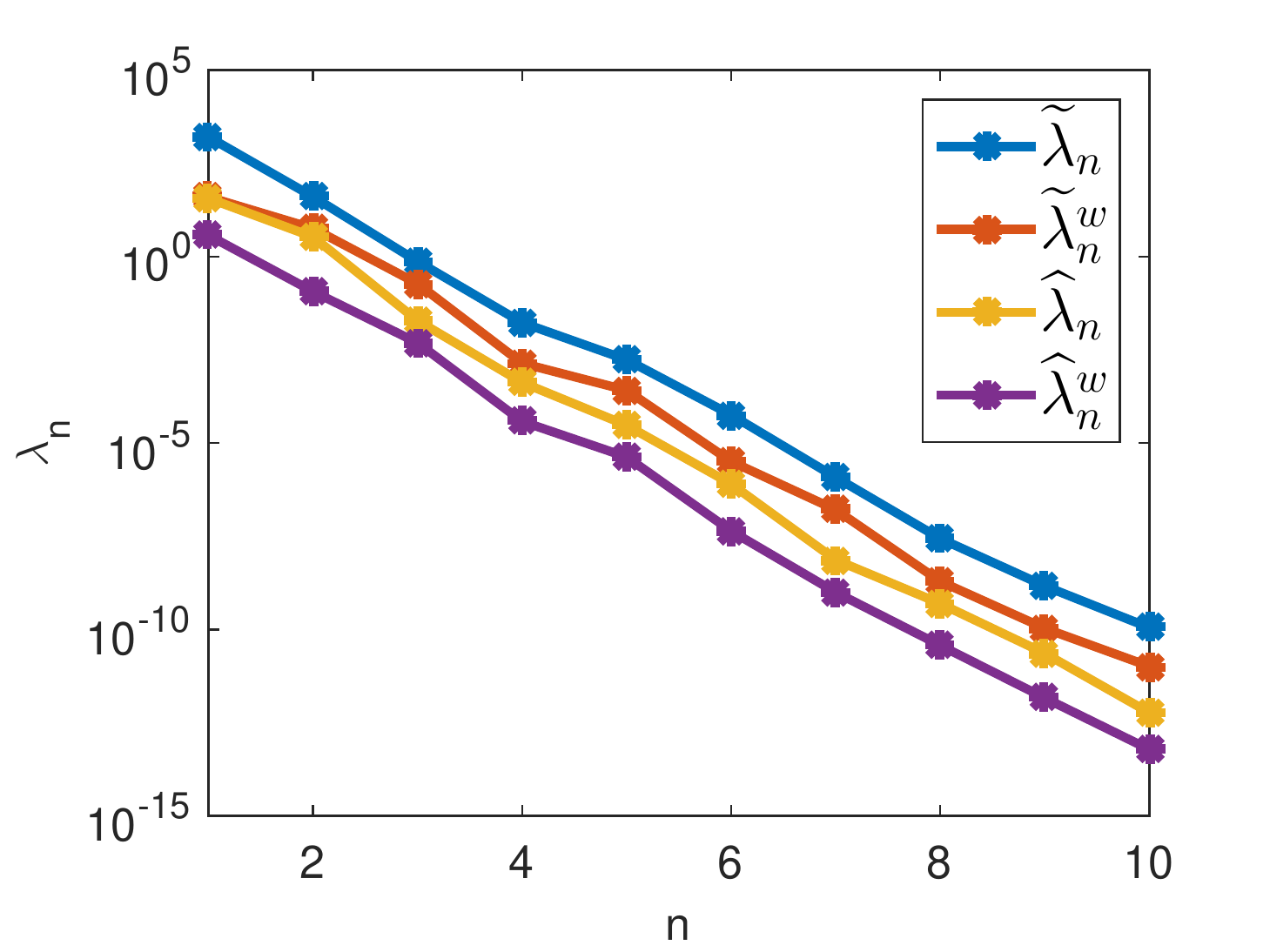}
}
\subfigure[$\al=0.5$]{
\includegraphics[trim = .1cm .1cm .1cm .1cm, clip=true,width=4.5cm]{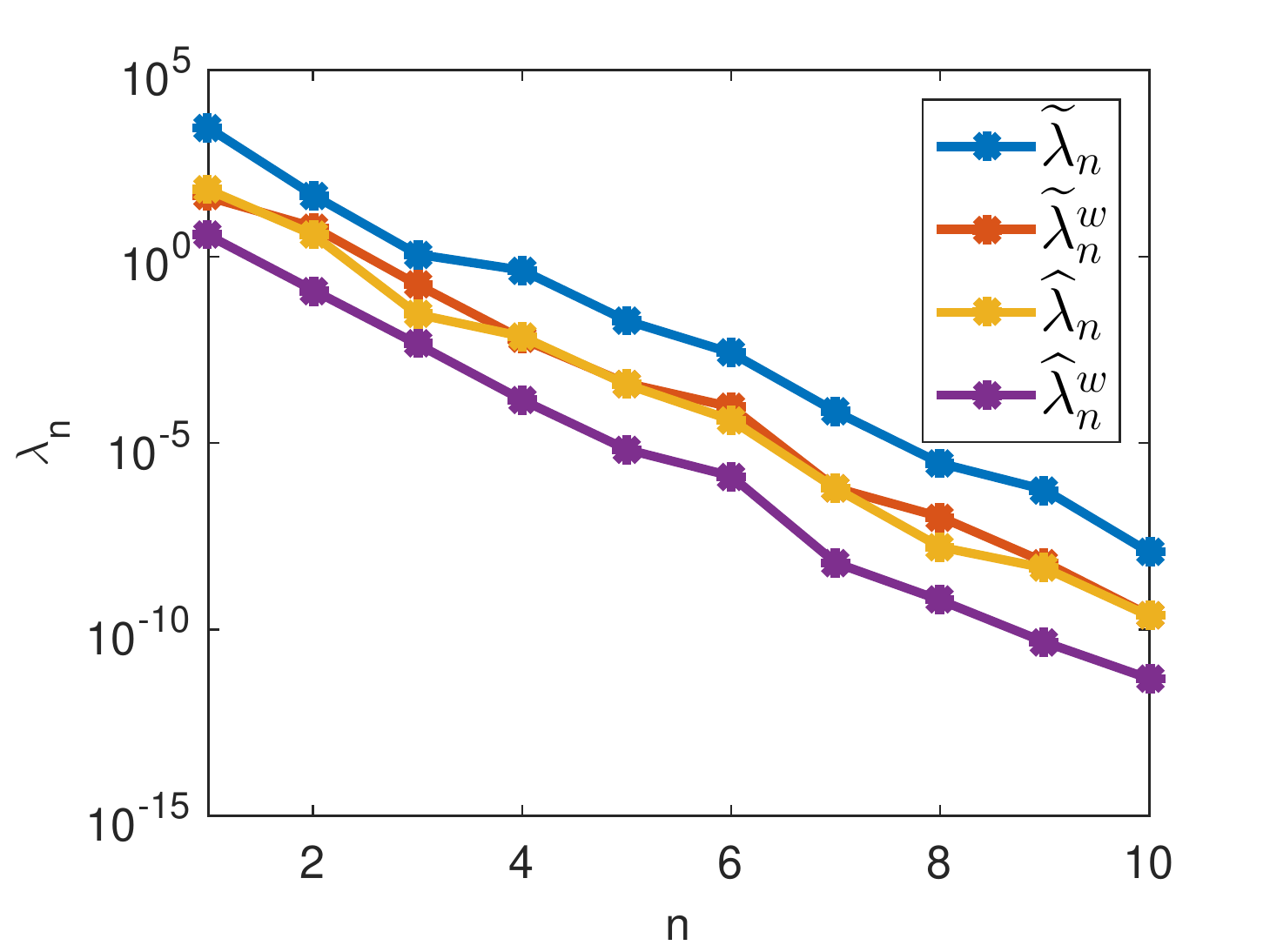}
}
\subfigure[$\al=0.7$]{
\includegraphics[trim = .1cm .1cm .1cm .1cm, clip=true,width=4.5cm]{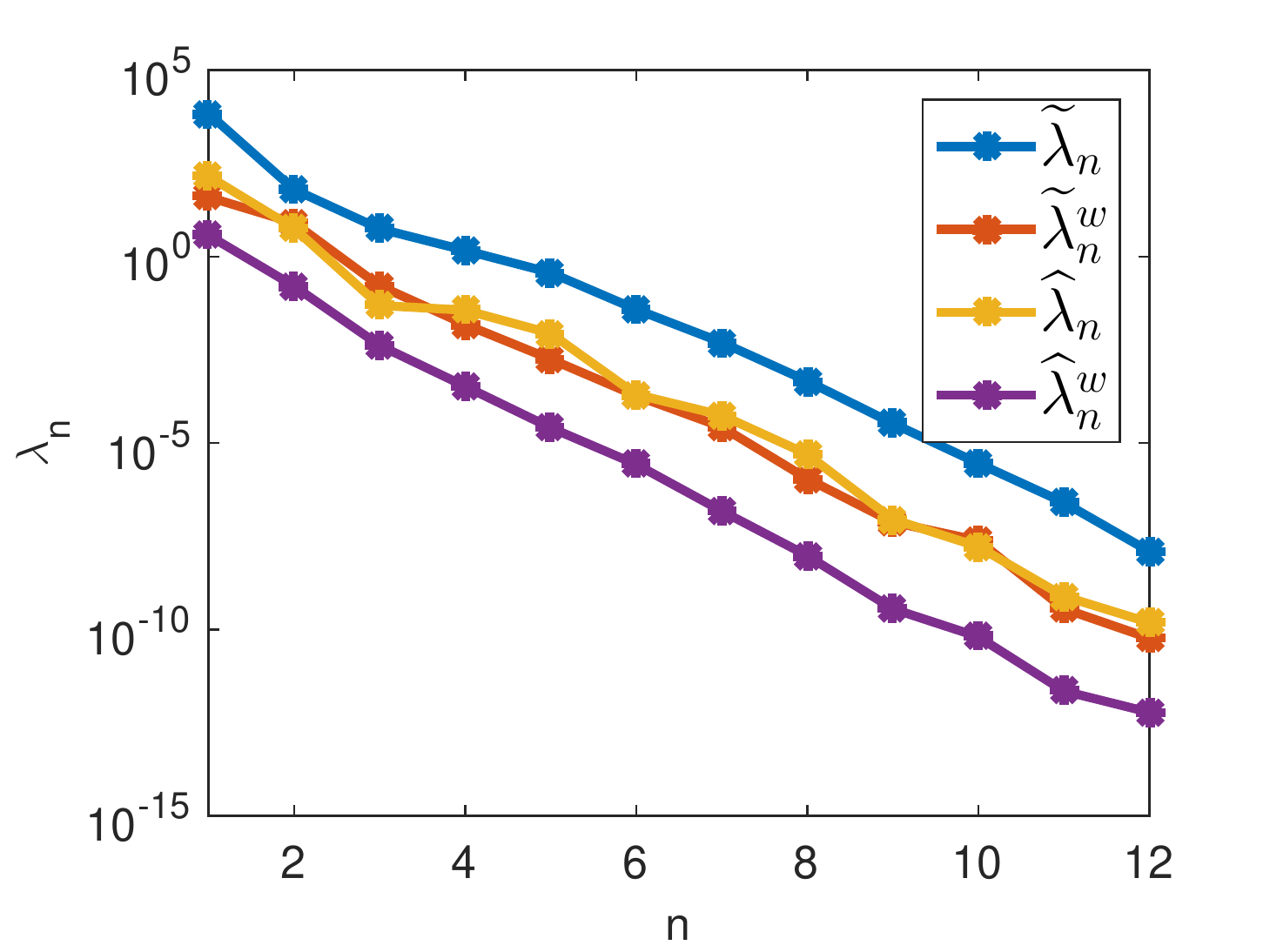}
}
 \caption{The decay of the eigenvalues of the correlation matrix for case (c) (2D problem on an L-shaped domain), with $\al=0.3$, $0.5$ and $0.7$.
 Here $\widetilde\lambda_n$, $\widetilde\lambda_n^w$, $\widehat\lambda_n$, and $\widehat\lambda_n^w$
 denote eigenvalues of correlation matrix for $H_0^1\II$ POD basis with or without FDQs
 and $L^2\II$ POD basis with or without FDQs,
 respectively.}\label{fig:POD-2D_lshape}
\end{figure}

\subsection{Numerical results for a perturbed problem}
Last, we illustrate the proposed Galerkin POD scheme with a perturbed problem, where the snapshots are
generated using a problem setting different from the one of interest, as typically occurs in optimal
control and inverse problems. Let $\delta_n(x)=n(2\cosh^2(nx))^{-1}$ be an approximate
Dirac delta function.
\begin{itemize}
  \item[(d)] On the domain $\Omega$ is $\Omega=(0,1)^2$, we consider the following problem:
  \begin{equation*}
    \Dal u -\Delta u + qu = f\quad \text{in }~\Omega
\end{equation*}
with $q(x_1,x_2)=1+\cos(\pi x_1) \sin(2 \pi x_2)$, $f(x_1,x_2,t)=\delta_2(x_1-\frac12)\delta_2(x_2-\frac12)e^{\cos(t)}$
 and $v(x_1,x_2)=x_1(1-x_1)\sin(2\pi x_2)$
and $T=1$. However, the snapshots are generated using a
perturbed source term $\tilde f(x_1,x_2,t)=\delta_{10}(x_1-\frac12)\delta_{10}(x_2-\frac12)$.
\end{itemize}
In our computation, we divide the sides of the domain $\Om$ into $100$ equal subintervals, each of
length $10^{-2}$, thus dividing $\Omega$ into $10^4$ small squares, and obtain a uniform
triangulation by connecting parallel diagonals of each small square. The time step size $\tau$ is fixed as $\tau=T/200$.

\begin{table}[hbt!]
\caption{The numerical results of the Galerkin POD for case (d),  with $T=1$, $N=200$
and with $m$ POD basis functions.}\label{tab:POD-d}
\vspace{-.2cm}
\begin{center}
\begin{tabular}{|c|c|c|c|c|c|}
\hline
   $\al$  & $m$ & ${\widetilde e}^m$ & ${\widetilde e}^m_w$ &$\widehat e^m$ &$\widehat e^m_w$ \\ \hline
 0.3 &4  & 4.63e-7 & 4.64e-7 &4.63e-7  &4.64e-7  \\
 & 5  & 3.32e-7 & 4.50e-7 &3.21e-7  &3.34e-7 \\  \hline
 0.5  &4  &4.47e-7 & 4.52e-7 &4.47e-7  &4.53e-7  \\
 & 5   &3.50e-7 & 3.46e-7 &3.50e-8  &3.45e-7 \\  \hline
 0.4  &4  &4.12e-7 & 4.32e-7&4.12e-7  &4.32e-7  \\
 & 5  & 3.81e-7 & 3.71e-7 &3.80e-7  &3.71e-7  \\  \hline
\end{tabular}
\end{center}
\end{table}

Since the snapshots are generated from a perturbed problem, the error estimates in Theorem \ref{thm:err-pod}
do not apply directly. Nonetheless, one can still observe a fast decay of the POD approximation error, and with
four to five POD basis functions, the error is already much smaller than the L1 time stepping, cf. Table
\ref{tab:POD-d}, for both $L^2(\Omega)$- and $H_0^1(\Omega)$-POD basis and with/without FDQs. The high
efficiency of the proposed scheme is attributed to the intrinsic low-dimensionality of the solution manifold, which is
fully captured by the snapshots generated from the perturbed problem. This is also
expected from the fast decay of the eigenvalues of the correlation matrix (from the perturbed problem) in Fig.
\ref{fig:POD-2D_sq}. The solution profiles and corresponding errors are shown in Fig. \ref{fig:POD-2D_sq_sol}.
This example  shows clearly the potential of the proposed approach for solving related inverse problems
and optimal control, where many analogous forward problems have to be solved.

\begin{figure}[hbt!]
\centering

\subfigure[$\al=0.3$]{
\includegraphics[trim = .1cm .1cm .1cm .1cm, clip=true,width=4.5cm]{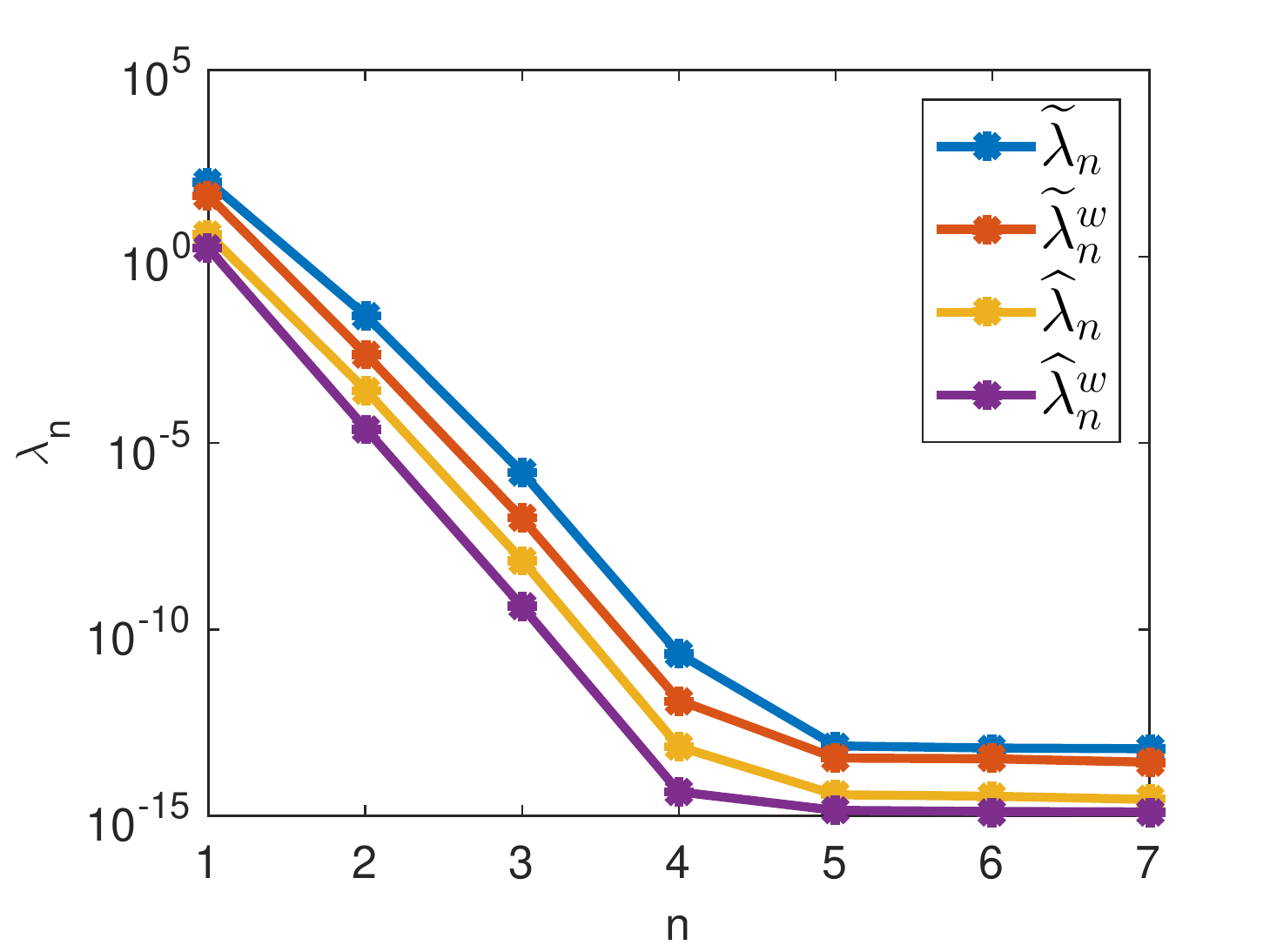}
}
\subfigure[$\al=0.5$]{
\includegraphics[trim = .1cm .1cm .1cm .1cm, clip=true,width=4.5cm]{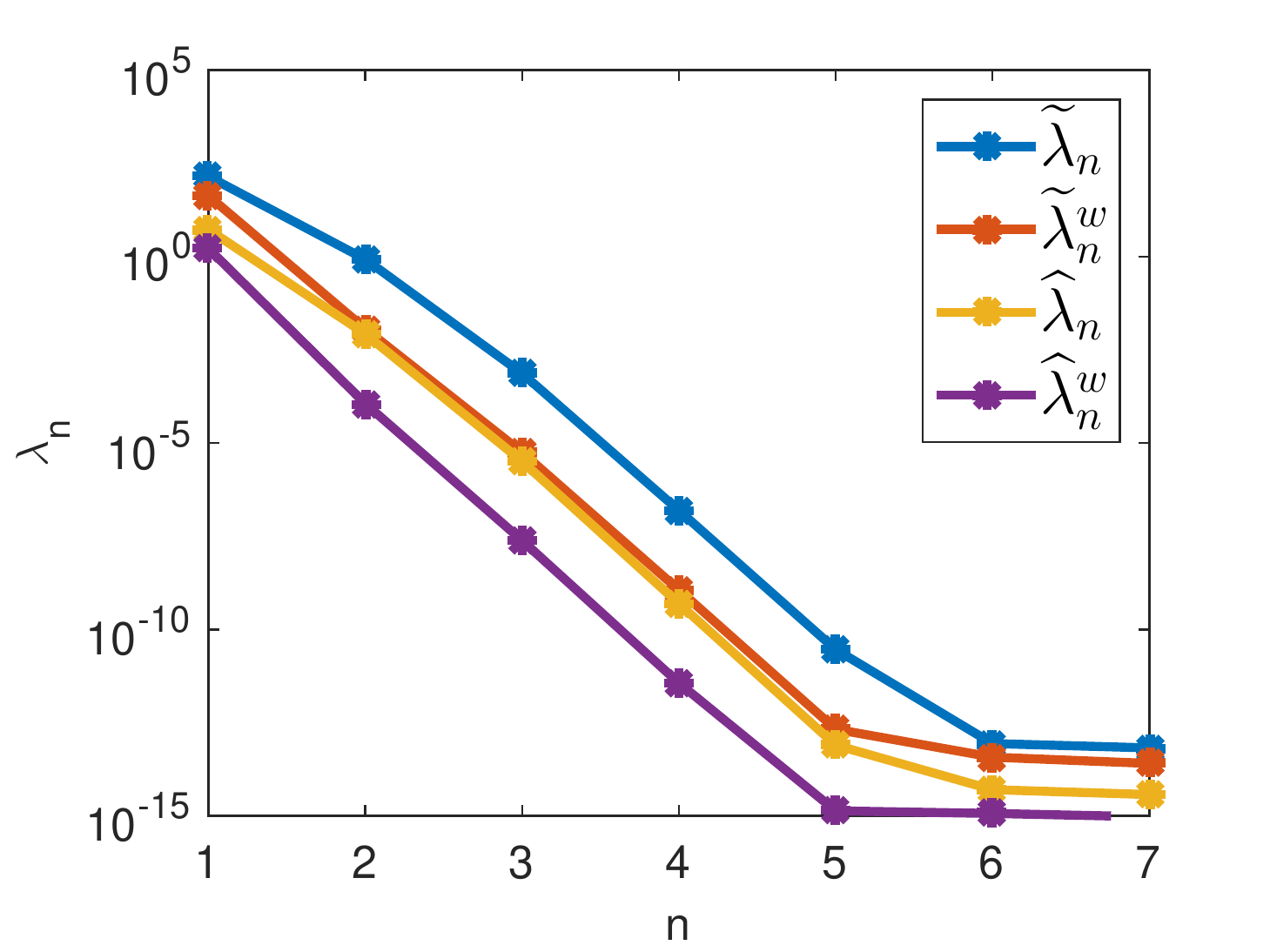}
}
\subfigure[$\al=0.7$]{
\includegraphics[trim = .1cm .1cm .1cm .1cm, clip=true,width=4.5cm]{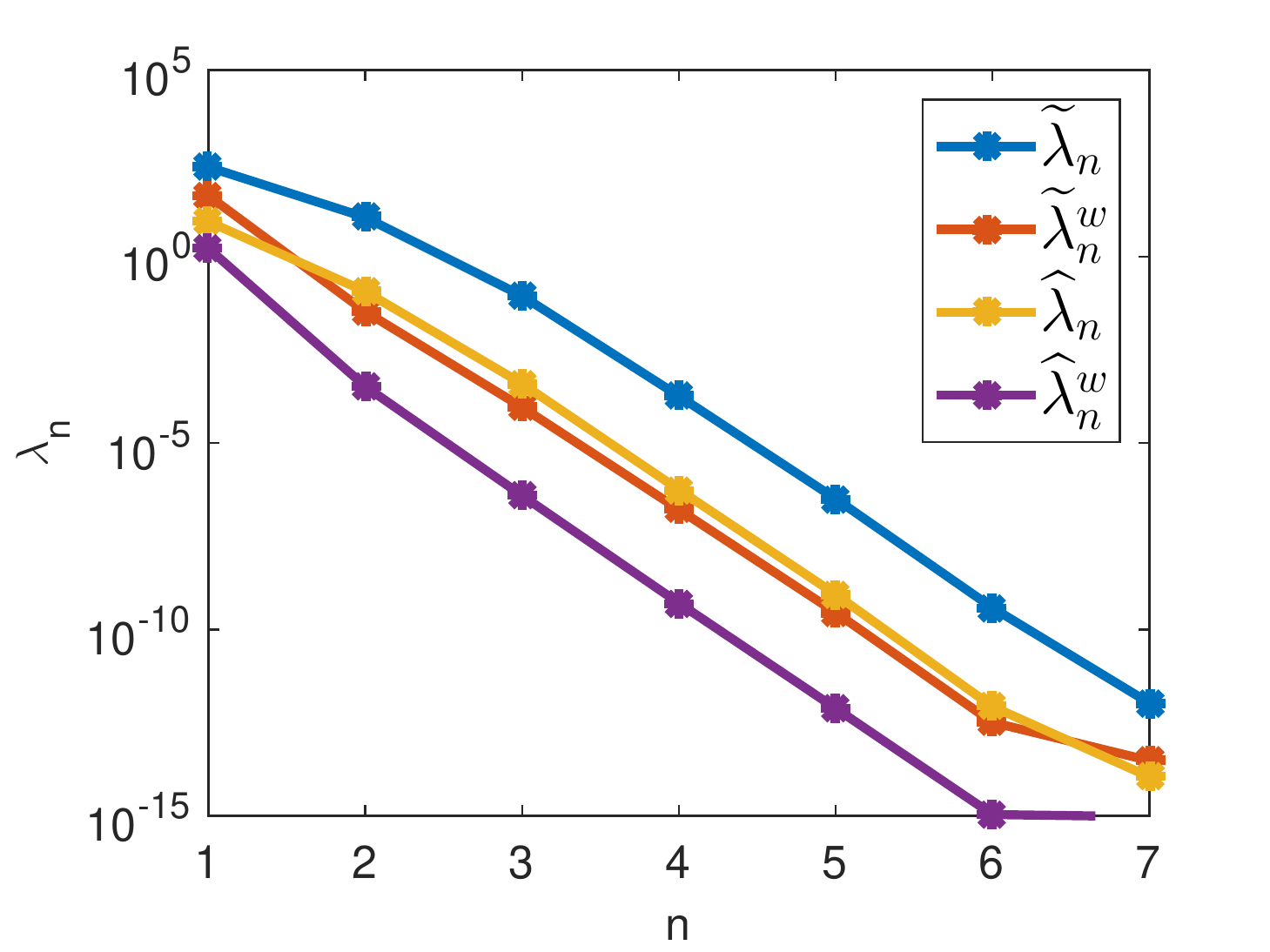}
}
 \caption{The decay of the eigenvalues of the correlation matrix for case (d) with $\al=0.3$, $0.5$ and $0.7$.
 Here $\widetilde\lambda_n$, $\widetilde\lambda_n^w$, $\widehat\lambda_n$, and $\widehat\lambda_n^w$
 denote eigenvalues of correlation matrix for $H_0^1\II$ POD basis with or without FDQs
 and $L^2\II$ POD basis with or without FDQs,
 respectively.}\label{fig:POD-2D_sq}
\end{figure}

\begin{figure}[hbt!]
\centering

\subfigure[exact solution, $\al=0.3$]{
\includegraphics[trim = .1cm .1cm .1cm .1cm, clip=true,width=4.5cm]{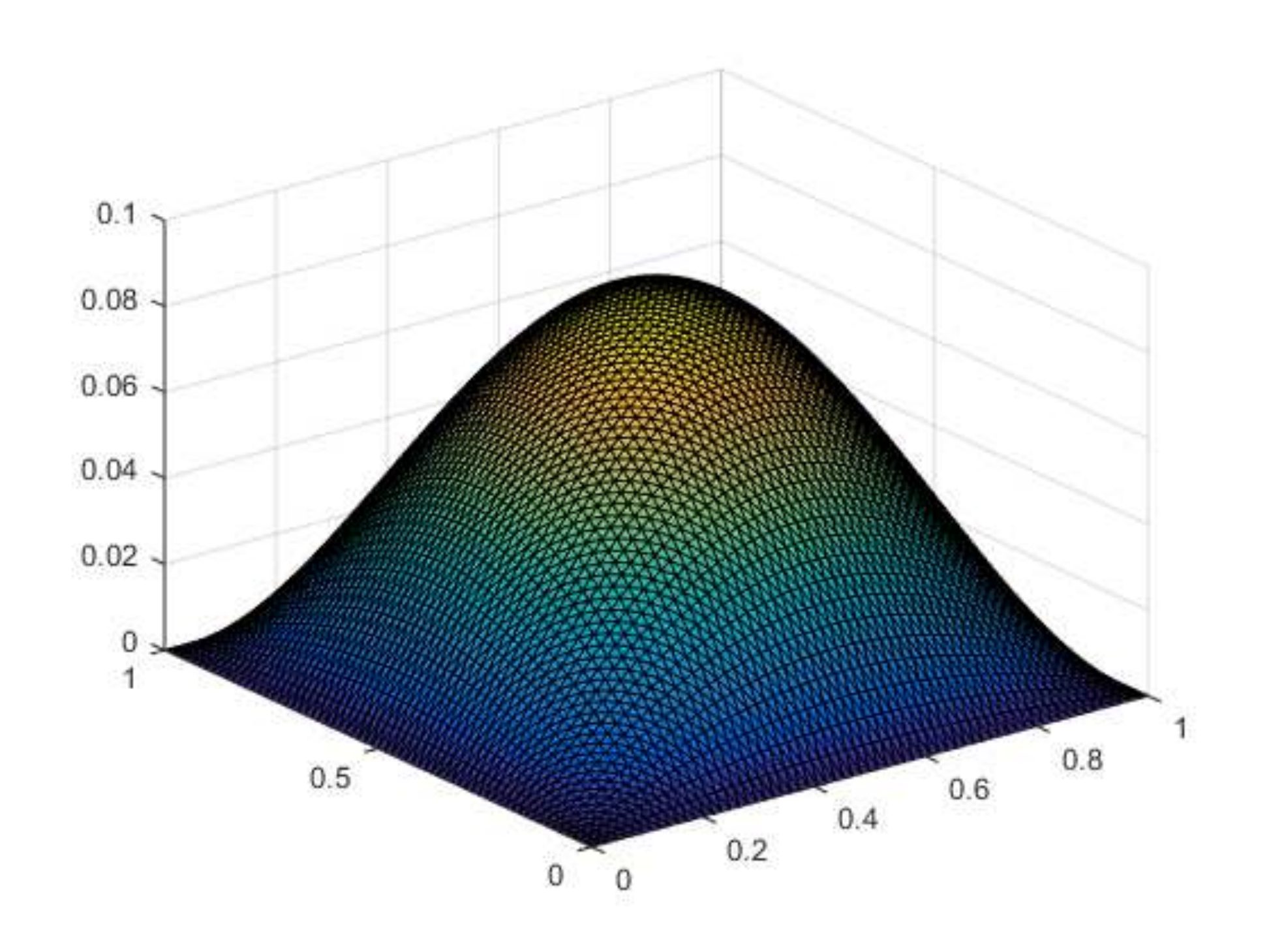}
}
\subfigure[POD solution, $\al=0.3$]{
\includegraphics[trim = .1cm .1cm .1cm .1cm, clip=true,width=4.5cm]{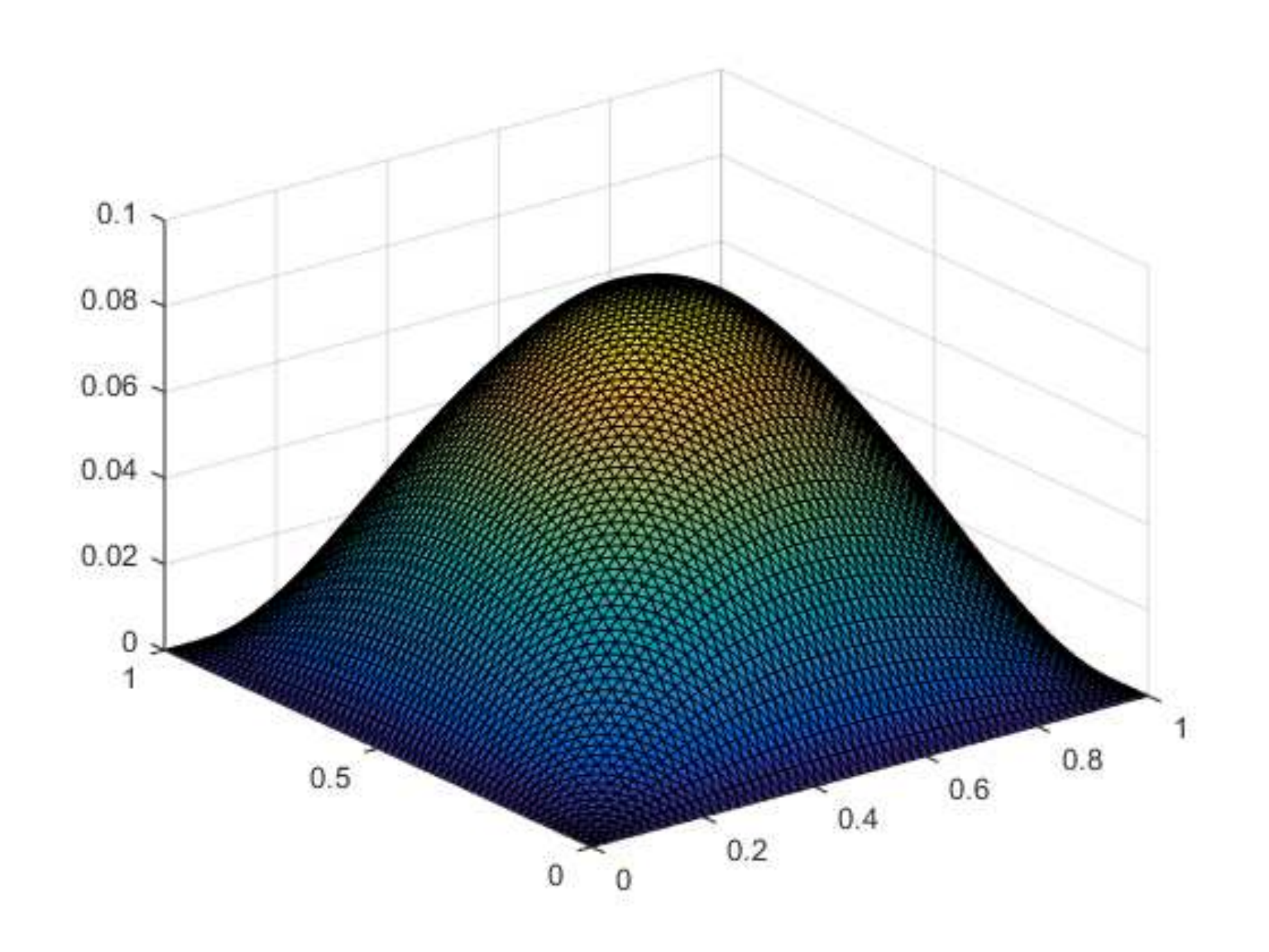}
}
\subfigure[error, $\al=0.3$]{
\includegraphics[trim = .1cm .1cm .1cm .1cm, clip=true,width=4.5cm]{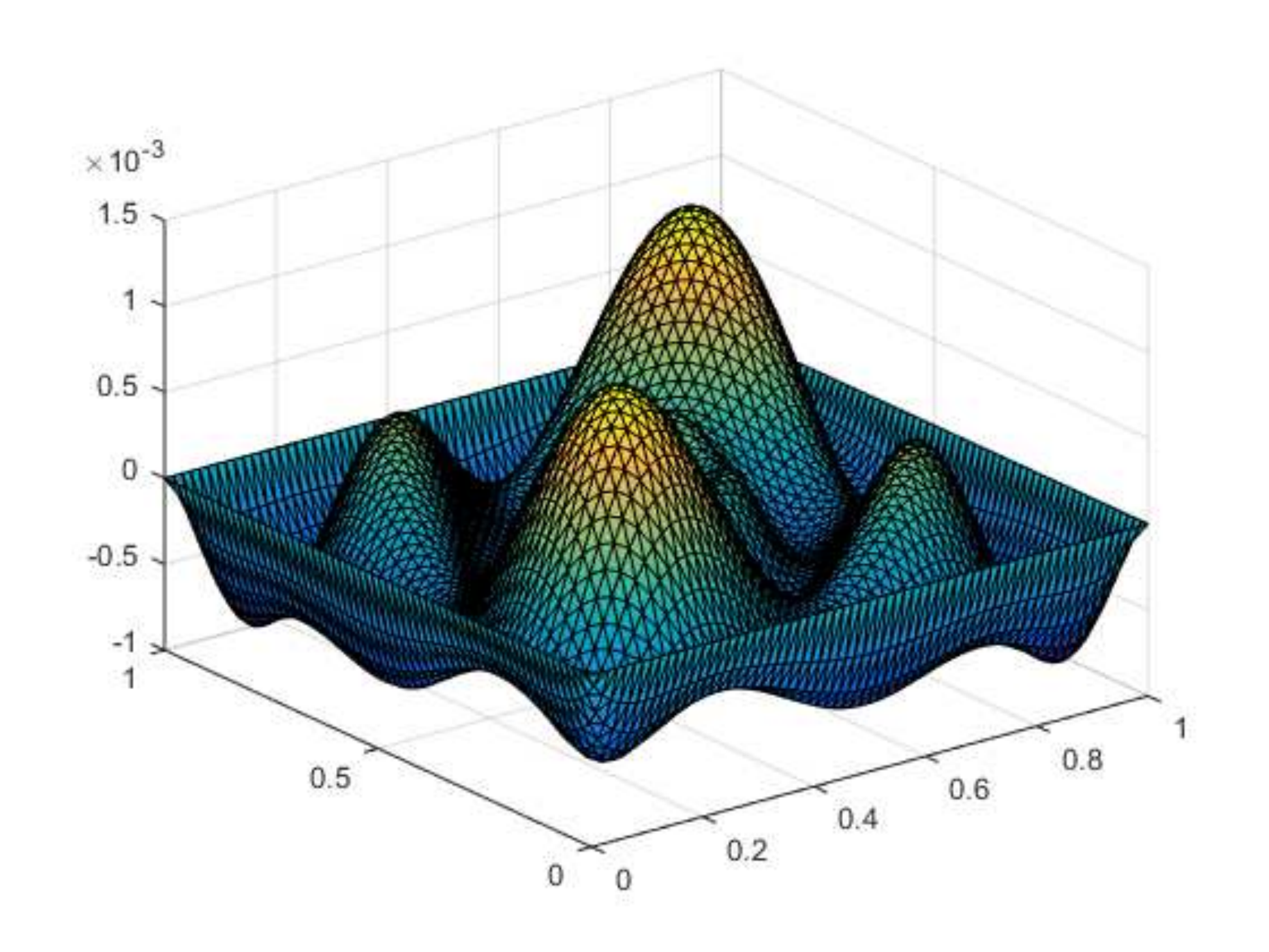}
}\\
\subfigure[exact solution, $\al=0.5$]{
\includegraphics[trim = .1cm .1cm .1cm .1cm, clip=true,width=4.5cm]{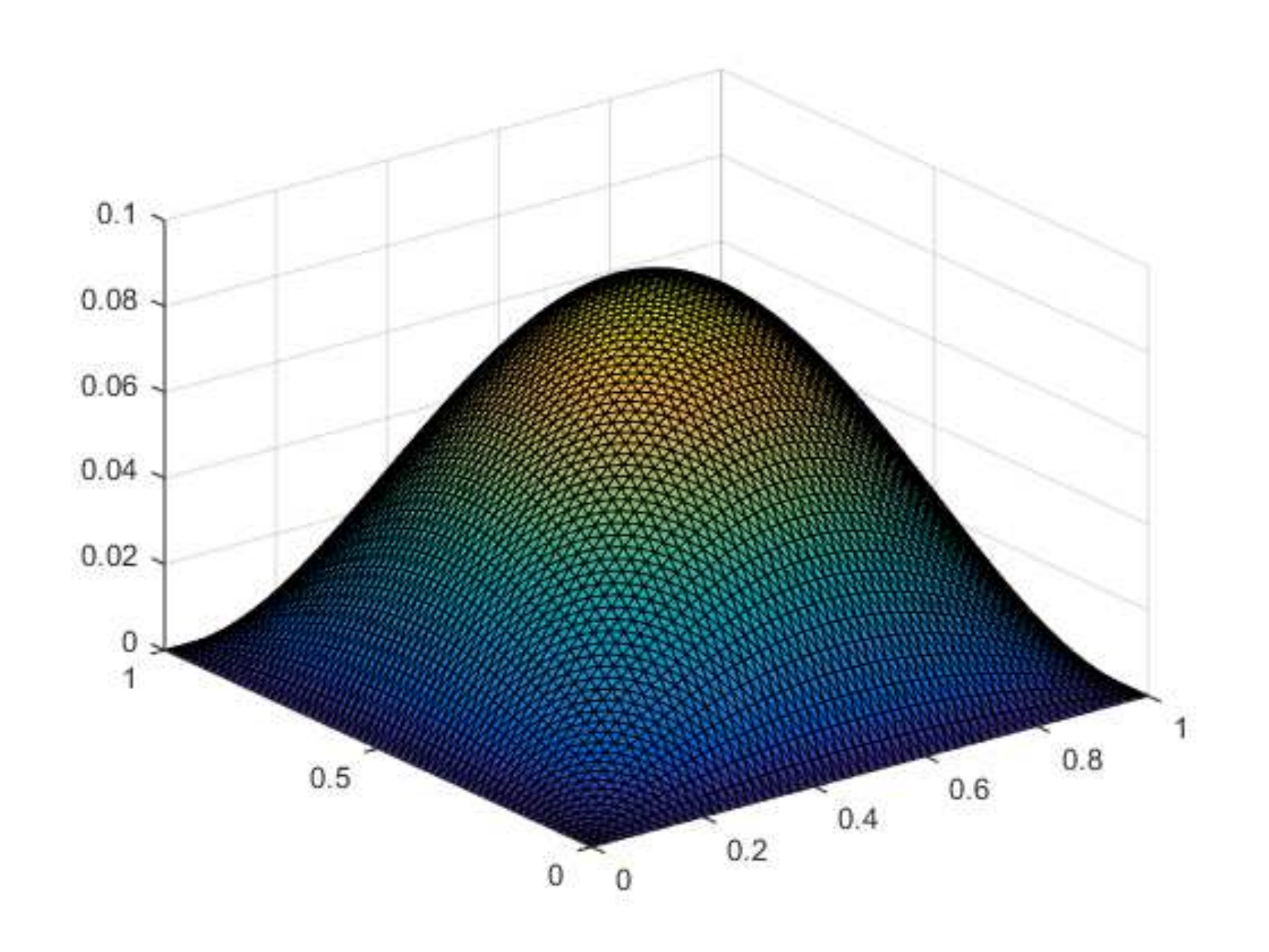}
}
\subfigure[POD solution, $\al=0.5$]{
\includegraphics[trim = .1cm .1cm .1cm .1cm, clip=true,width=4.5cm]{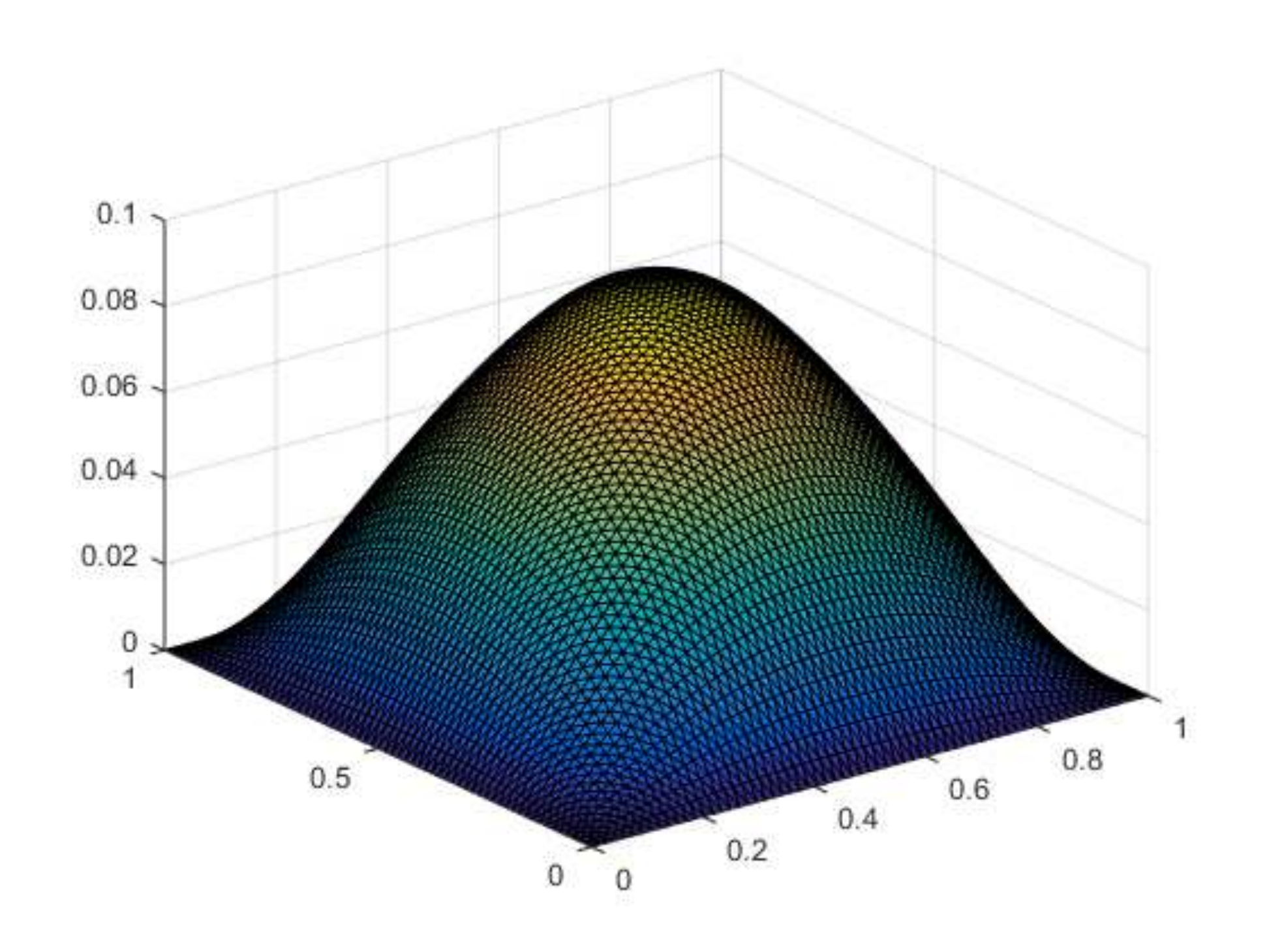}
}
\subfigure[error, $\al=0.5$]{
\includegraphics[trim = .1cm .1cm .1cm .1cm, clip=true,width=4.5cm]{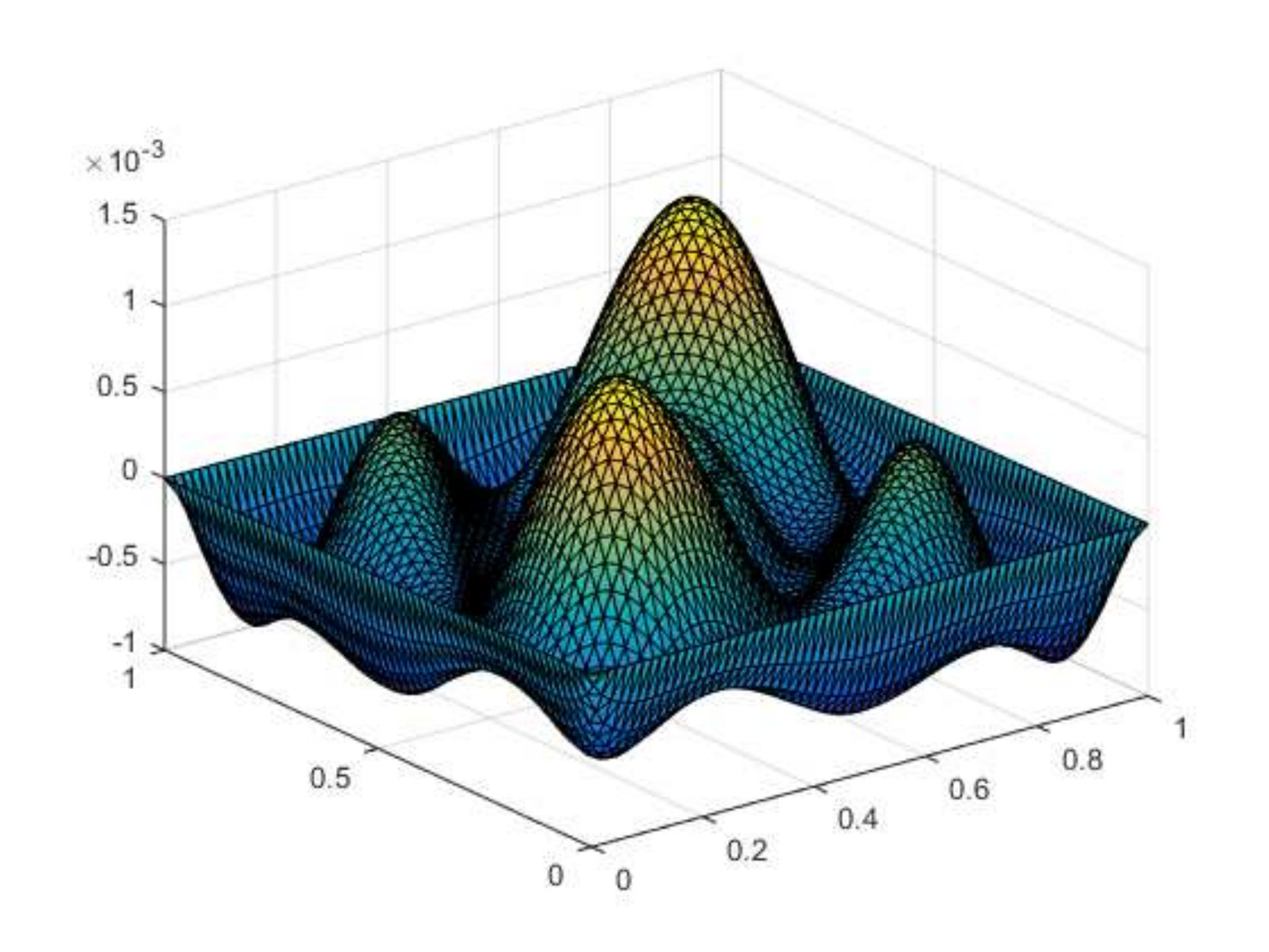}
}\\
\subfigure[exact solution, $\al=0.7$]{
\includegraphics[trim = .1cm .1cm .1cm .1cm, clip=true,width=4.5cm]{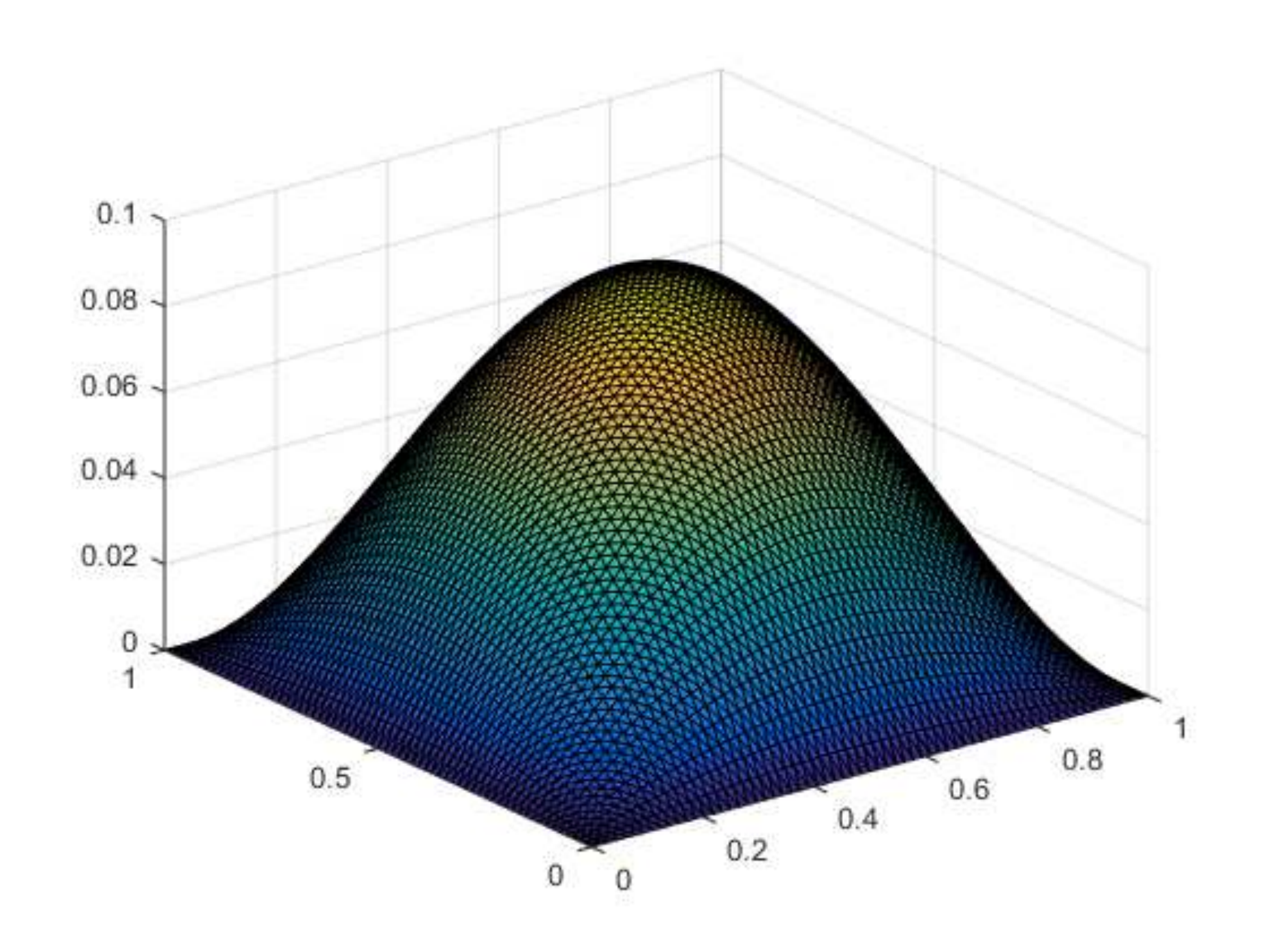}
}
\subfigure[POD solution, $\al=0.7$]{
\includegraphics[trim = .1cm .1cm .1cm .1cm, clip=true,width=4.5cm]{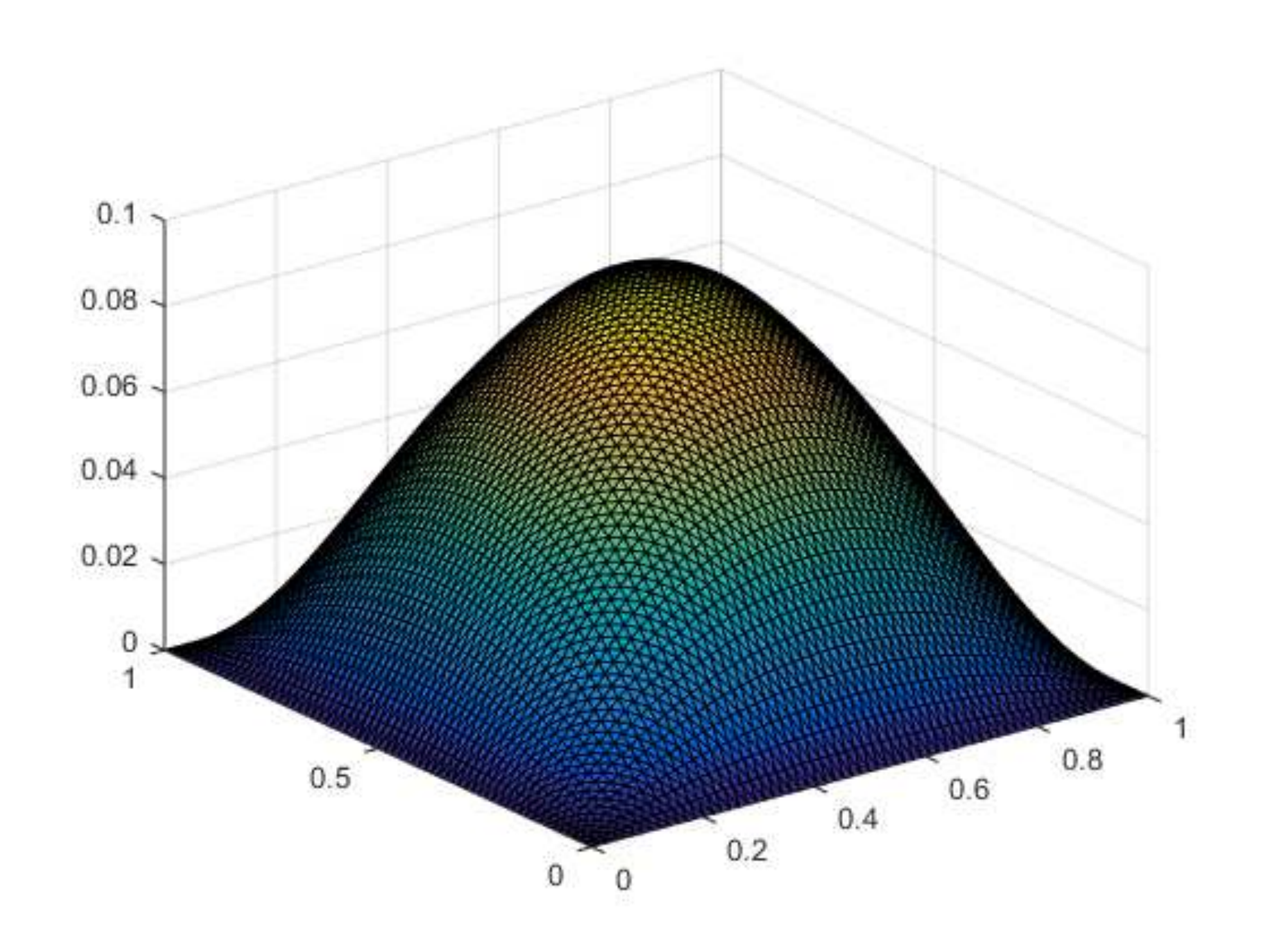}
}
\subfigure[error, $\al=0.7$]{
\includegraphics[trim = .1cm .1cm .1cm .1cm, clip=true,width=4.5cm]{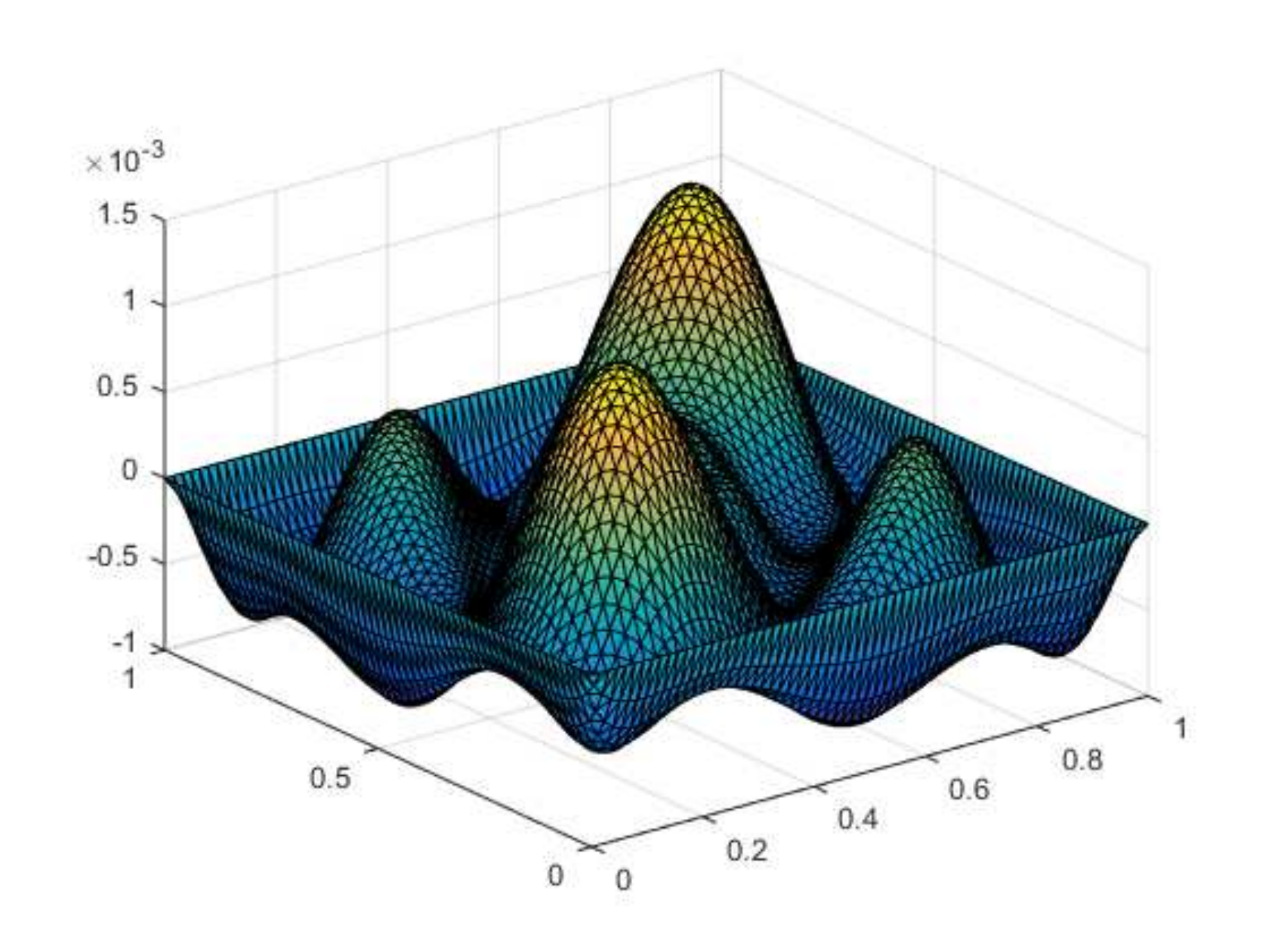}
}\\
 \caption{Exact and numerical solutions {at $T=1$} for case (d), where the POD solutions
are obtained using $H_0^1(\Omega)$ POD basis with the FDQs.}\label{fig:POD-2D_sq_sol}
\end{figure}

\section{Concluding remarks}
In this work, we have developed an efficient Galerkin-L1-POD scheme for solving the subdiffusion
problem, by coupling the Galerkin finite element method, L1 time stepping and proper
orthogonal decomposition. It realizes the computational efficiency by constructing an effective
reduced-order model using POD, often with a very small degree of freedom.
We provided a complete error analysis of the scheme, and derived optimal error estimates due
to spatial discretization, temporal discretization and POD approximation. This is achieved
by developing a novel energy argument for L1 time stepping. The extensive numerical experiments
fully confirmed the convergence analysis and the efficiency and robustness of the scheme.

The work represents only a first step towards effective model reduction strategies for fractional differential equations.
The choice of the three components in the proposed scheme is not unique. Alternatively, one may employ
finite difference methods or spectral methods instead of the finite element method, and
convolution quadrature type schemes instead of the L1 time scheme.
The overall framework extends straightforwardly to these alternative choices, even though the
convergence analysis will differ. Further, it is of much interest to extend the proposed scheme
to more complex models, e.g., the multi-term model and the diffusion-wave model.

\section*{Acknowledgements}
The work of the first author (B. Jin) is partly supported by EPSRC grant EP/M025160/1.

\appendix
\section{Regularity theory for problem \eqref{eqn:fde}}
Now we describe temporal regularity results of problem \eqref{eqn:fde} which plays an important role
in the convergence analysis. Let $\{(\lambda_j,\fy_j)\}_{j=1}^\infty$ be the eigenvalue pairs of the
negative Laplacian $A=-\Delta$ with a homogeneous Dirichlet boundary condition, where the set $\{\fy_j
\}_{j=1}^\infty$ forms an orthonormal basis in $L^2\II$. Then by the standard separation of variable
technique, we deduce that the solution $u$ can be represented by
\begin{equation*}
  u(t)=E(t)v + \int_0^t  {\bar E}(t-s) f(s) ds,
\end{equation*}
where the solution operators $E(t)$ and $\bar E(t) $ are given by
\begin{equation}\label{op:E}
  E(t)\psi=\sum_{j=1}^\infty E_{\alpha,1}(-\lambda_jt^\alpha)(\psi,\fy_j)\fy_j\quad
  \text{and}\quad {\bar E(t)}\psi =\sum_{j=1}^\infty t^{\alpha-1}E_{\alpha,\alpha}(-\lambda_jt^\alpha)(\psi,\fy_j)\fy_j,
\end{equation}
respectively. 
Here the Mittag-Leffler
function $E_{\alpha,\beta}(z)$, $\alpha>0$, $\beta\in\mathbb{R}$, is defined
by \cite[pp. 42]{KilbasSrivastavaTrujillo:2006}
$
  E_{\alpha,\beta}(z) = \sum_{k=0}^\infty \frac{z^k}{\Gamma(k\alpha +\beta)}.
$
The following relations hold
(see \cite[Lemma 3.2]{SakamotoYamamoto:2011} and \cite[pp. 43, eq. (1.8.28)]{KilbasSrivastavaTrujillo:2006} for proofs).
\begin{lemma}\label{lem:mlfbdd}
Let $\alpha\in(0,1)$, and $\beta\in\mathbb{R}$. The Mittag-Leffler function $E_{\alpha,\beta}(z)$ satisfies for $m\geq1$
\begin{equation*}
  \frac{d^m}{dt^m} E_{\alpha,1}(-\lambda t^\alpha) = -\lambda t^{\alpha-m} E_{\alpha,\alpha+1-m}(-\lambda t^\alpha)\quad t>0,
\end{equation*}
and the following uniform bound on the negative real axis $\mathbb{R}^-$ holds
\begin{equation*}
  E_{\alpha,\beta}(z) \leq   c(1+|z|)^{-1}\quad \forall z\in \mathbb{R}^{-}.
\end{equation*}
\end{lemma}

Now we can state the temporal regularity for the homogeneous problem.
\begin{theorem}\label{thm:reg-init}
If $v\in D(A^\sigma)$ and $f\equiv 0$, then
  \begin{equation}\label{eqn:new1}
   \| \partial_t^m u \|_{L^2(\Om)} \le c t^{\sigma\al-m}  \| A^\sigma v \|_{L^2\II},
  \end{equation}
  where if $\sigma\in(0,1]$, $m\geq1$, and if $\sigma=0$, $0\leq m\leq 2$.
\end{theorem}
\begin{proof}
The case $\sigma=0$ has been shown \cite[Corollary 2.6]{SakamotoYamamoto:2011}.
For $\sigma\in(0,1]$, by Lemma \ref{lem:mlfbdd}, we have
\begin{equation*}
  \begin{aligned}
   \|\partial_t^m u\|_{L^2(\Omega)}^2 & = \|\sum_{j=1}^\infty \frac{d^m}{dt^m}E_{\alpha,1}(-\lambda_jt^\alpha)(v,\fy_j)\fy_j\|_{L^2(\Omega)}^2
      =\sum_{j=1}^\infty \lambda_j^2 t^{2\alpha-2m}E_{\alpha,\alpha-m+1}(-\lambda_jt^\alpha)^2(v,\fy_j)^2\\
     & =\sum_{j=1}^\infty (\lambda_j t^\alpha)^{2-2\sigma} t^{2\sigma\alpha-2m}E_{\alpha,\alpha-m+1}(-\lambda_jt^\alpha)^2(v,\fy_j)^2\lambda_j^{2\sigma}\\
     & \leq ct^{2\sigma\alpha-2m}\sup_j \frac{(\lambda_jt^{\alpha})^{2-2\sigma}}{(1+\lambda_jt^\alpha)^2}\sum_{j=1}^\infty (v,\fy_j)^2\lambda_j^{2\sigma}\leq ct^{2\sigma\alpha-2m}\|A^\sigma v\|_{L^2\II}^2,
  \end{aligned}
\end{equation*}
where the last inequality follows from the inequality $\sup_j {(\lambda_jt^{\alpha})^{2-2\sigma}}/{(1+\lambda_jt^\alpha)^2}\leq c$.
\end{proof}

Next we consider the inhomogeneous problem. We shall need the following estimate on $\bar E(t)$
\begin{lemma}\label{lem:Ebar}
For any $t>0$, we have for $ \chi\in L^2\II$ and  $m\ge 0$
\begin{equation*}
     \|\partial_t^m\bar E(t) \chi \|_{L^2 \II} \le ct^{\al-m-1}\|\chi\|_{L^2\II}.
\end{equation*}
\end{lemma}
\begin{proof}
The definition of the operator $\bar{E}$ in \eqref{op:E} and Lemma \ref{lem:mlfbdd} yield
\begin{equation*}
\begin{split}
\|\partial_t^m \bar E(t) \chi \|_{L^2\II}^2 & =\sum_{j=1}^{\infty}  |t^{\al-m-1}E_{\al,\al-m}(-\la_j t^\al)|^2 |(\chi,\fy_j)|^2\\
     & \le c t^{2\al-2m-2}\sum_{j=1}^{\infty}  |(\chi,\fy_j)|^2 = c t^{2\al-2m-2}  \|\chi\|_{L^2\II}^2,
\end{split}
\end{equation*}
which completes the proof of the lemma.
\end{proof}

Now we can state the temporal regularity result for the inhomogeneous problem.
\begin{theorem}\label{thm:reg}
If $v\equiv 0$ and $f\in W^{m,\infty}(0,T;L^2(\Omega))$ with some $m \in [0,2]$, then there holds
  \begin{equation}\label{eqn:new2}
   \| \partial_t^m u\|_{L^2(\Om)} \le c_T t^{\al-m}  \| f\|_{W^{m,\infty}(0,T;L^2\II)}, \quad 0\le m\le 2.
  \end{equation}
\end{theorem}
\begin{proof}
Using the following convolution relation \cite[Lemma 5.2]{McLean:2010}
\begin{equation*}
  t(f*g)' = f*g+(tf')*g+f*(tg')
\end{equation*}
and Lemma \ref{lem:Ebar}, we deduce that for $t\in(0,T]$
\begin{equation*}
\begin{split}
  t^m \| \partial_t^m u\|_{L^2(\Om)}
 & \le c\sum_{p+q\le m}\int_0^t  \|(t-s)^p\partial_t^p \bar E(t-s)  (s^{q}f^{(m)}(s))\|_{L^2\II} \,ds\\
  & \le  c\sum_{p+q\le m}\int_0^t  (t-s)^{\al-1} s^{q}  \|f^{(m)}(s))\|_{L^2\II} \,ds
  \le  c \|f\|_{W^{m,\infty}(0,T;L^2\II)} \sum_{p+q\le m}  t^{\al+q} .
\end{split}
\end{equation*}
Since for $t\in (0,T]$, we have $\sum_{p+q\le m}  t^{\al+q-m} \le c_T t^{\al-m}$,
the desired assertion follows.
\end{proof}

\begin{remark}\label{rmk:reg}
Theorems \ref{thm:reg-init} and \ref{thm:reg} show the limited
smoothing property of the subdiffusion model \eqref{eqn:fde}:
for the homogeneous problem with $v\in D(A)$, the first
order derivative in time $t$ of the solution $u$ exhibits a singularity of the form $t^{\alpha-1}$; and for
the inhomogeneous problem with $f\in W^{2,\infty}(0,T;L^2(\Omega))$, the first-order derivative
exhibits a similar singularity, despite the smoothness of $f$ in time.
\end{remark}

\bibliographystyle{abbrv}
\bibliography{frac-l1}

\end{document}